\documentclass[10pt]{amsart}

\usepackage{amssymb}
\usepackage[leqno]{amsmath}
\usepackage{mathrsfs}
\usepackage{stmaryrd}
\usepackage{chemarrow}

\usepackage{algorithm}
\usepackage{algpseudocode}


\usepackage{enumitem}
\usepackage{graphicx}
\usepackage{caption,subcaption}
\usepackage{xcolor}
\usepackage[pdftex,bookmarksnumbered,bookmarksopen,colorlinks,linkcolor=red,anchorcolor=black,citecolor=blue,urlcolor=blue]{hyperref}
\usepackage[all]{xy}
\usepackage{tikz}
\usetikzlibrary{arrows}

\usepackage{multirow}
\usepackage[hyperpageref]{backref}

  \newcounter{mnote}
  \let\oldmarginpar\marginpar
    \renewcommand\marginpar[1]{\-\oldmarginpar[\raggedleft\footnotesize #1]%
    {\raggedright\footnotesize #1}}

\newtheorem{theorem}{Theorem}[section]
\newtheorem{lemma}[theorem]{Lemma}

\newtheorem{remark}[theorem]{Remark}

\newcommand{\dd}{\,{\rm d}}
\newcommand{\bs}{\boldsymbol}

\newcommand{\curl}{\operatorname{curl}}
\renewcommand{\div}{\operatorname{div}}
\newcommand{\grad}{\operatorname{grad}}
\newcommand{\gradcurl}{\operatorname{grad}\operatorname{curl}}
\newcommand{\curlcurl}{\operatorname{curl}\operatorname{curl}}
\DeclareMathOperator*{\tr}{tr}
\DeclareMathOperator*{\rot}{rot}

\numberwithin{equation}{section}

\makeatletter
\@namedef{subjclassname@2020}{%
  \textup{2020} Mathematics Subject Classification}
\makeatother

\begin{document}
\title[Decoupled FEM for quad-curl problems]{Error analysis of a decoupled finite element method for quad-curl problems}

\author{Shuhao Cao}%
\address{Department of Mathematics and Statistics, Washington University, St. Louis, MO 63130, USA}%
\email{s.cao@wustl.edu}%
\author{Long Chen}%
\address{Department of Mathematics, University of California at Irvine, Irvine, CA 92697, USA}%
\email{chenlong@math.uci.edu}%
\author{Xuehai Huang}%
\address{School of Mathematics, Shanghai University of Finance and Economics, Shanghai 200433, China}%
\email[Corresponding author]{huang.xuehai@sufe.edu.cn}%

\thanks{The third author is the corresponding author.}
\thanks{The first author was supported in part by the National Science Foundation under grant DMS-1913080. The second author was supported in part by the National Science Foundation under grants DMS-1913080 and DMS-2012465. The third author was supported by the National Natural Science Foundation of China under grants 11771338 and 12171300, the Natural Science Foundation of Shanghai 21ZR1480500, and the Fundamental Research Funds for the Central Universities 2019110066.}


\begin{abstract}
Finite element approximation to a decoupled formulation for the quad--curl problem is studied in this paper. The difficulty of constructing elements with certain conformity to the quad--curl problems has been greatly reduced. For convex domains, where the regularity assumption holds for Stokes equation, the approximation to the curl of the true solution has quadratic order of convergence and first order for the energy norm. If the solution shows singularity, an a posterior error estimator is developed and a separate marking adaptive finite element procedure is proposed, together with its convergence proved. Both the a priori and a posteriori error analysis are supported by the numerical examples.
\end{abstract}


\maketitle

%\tableofcontents

\section{Introduction}
Quad-curl problem arises from multiphysics simulation such as modeling a
magnetized plasma in magnetohydrodynamics (MHD). In both limiting regimes, resistive MHD (\cite{biskamp1996magnetic,Sweet:1958Neutral}) and electron MHD (\cite{kingsep1990reviews,Chacon;Simakov;Zocco:2007Steady-state,simakov2009quantitative}), discretizing the quad-curl operator is one of the keys to simulate these models. In the meantime, quad-curl operator also plays an important role in approximating the Maxwell transmission eigenvalue problem \cite{Monk;Sun:2012Finite,Cakoni;Haddar:2007variational}. Recently, the designing of the approximations for quad-curl problems gain quite a few attentions from the finite element community. 
For example, conforming finite element spaces for the quad-curl problem has been recently constructed in \cite{Zhang;Wang;Zhang:20192-Conforming,Hu;Zhang;Zhang:2020Simple} in two dimensions and \cite{Neilan:2015Discrete,Zhang;Zhang:2020Curl-curl,HuZhangZhang2020curcurl3d} in three dimensions. Nonconforming and low order finite element spaces can be found in \cite{Zheng;Hu;Xu:2011nonconforming,Huang2020}. The mixed methods are studied in \cite{Sun:2016quad-curl,Zhang2018a,WangSunCui2019new}. A formulation based on the Hodge decomposition is in \cite{BrennerSunSung2017Hodge}. A discontinuous Galerkin approach is studied in \cite{HongHuShuXu2012Discontinuous}. In \cite{SunZhangZhang2019curl}, a novel weak Galerkin formulation exploits the conforming space for curl-curl problem as a nonconforming space for the quad-curl problem. The a posteriori error analysis in two dimensions is studied in \cite{WangZhangSunZhang2020priori}. We also refer to \cite{ZhaoZhang2021} for a virtual element method in two dimensions.

The structures of the quad-curl problem are unique as the operator may have a bigger kernel than the one in the curl-curl problem. In a simply-connected domain, the weak formulation of the quad-curl problem is equivalent to that with a grad curl operator \cite{Zhang2018a,Arnold;Hu:2021Complexes}. Consequently, the stringent continuity condition of the grad curl drives the local polynomial space's dimension to be much bigger than that of the curl-curl problem. This poses extra difficulty in constructing the conforming finite element approximations, and renders them hard to solve especially in three dimensions. The nonconforming elements \cite{Zheng;Hu;Xu:2011nonconforming} greatly simply the local structure of the space, and is more preferable in approximating the quad-curl problem in terms of the computational resources.

In \cite{ChenHuang2018,Huang2020}, a novel way of further simplifying the structure of the quad-curl problem is proposed. The quad-curl problem is decoupled into three sub-problems, two curl-curl equations, and one Stokes equation, all of which have mature finite element approximation theories (e.g., \cite{CrouzeixRaviart1973,Nedelec1980,Kikuchi:1989discrete,Kikuchi:1989mixed}). In this paper, we use lower-order N\'ed\'elec elements \cite{Nedelec1980,Nedelec1986} to discretize curl-curl equations, and the nonconforming $P_1$-$P_0$ finite element to discretize Stokes equation, then analyze this decoupled finite element method (FEM) for the quad-curl problem. Due to the decoupling mechanism, one of the major advantages is that the curl of the primal variable can be approximated an order higher than most of the conforming or nonconforming FEMs.

Meanwhile, due to the nature of quad-curl operator \cite{Nicaise:2018Singularities}, on a polyhedral domain, the singularities of solution may manifest themselves as either corner singularities of the Stokes system with Dirichlet boundary conditions, corner/edge singularities of the Maxwell problem, or both. To cope with such solutions with the presence of singularities, adaptive finite element method (AFEM) is favored over the finite element method performed on a uniformly refined mesh. The computational resources are adaptively allocated throughout different locations of the mesh based on the local estimated approximation error. Thus the AFEM can achieve the same overall accuracy while using fewer degrees of freedom than the one with uniform mesh.

Opting for a decoupled system using existing and mature elements for each offers great facilitation to the AFEM pipeline. Now there are three major pieces to the puzzle: 
the a posteriori error estimation for the conforming approximation to the Maxwell problem 
(e.g., \cite{Beck;Hiptmair;Hoppe;Wohlmuth:2000Residual,
Cai;Cao:2015recovery-based,Schoberl2008,Cochez-Dhondt;Nicaise:2007Robust}), that for a nonconforming discretization to the Stokes problem (e.g., \cite{DariDuranPadra1995,Verfurth:1991posteriori,Dari;Duran;Padra;Vampa:1996posteriori}), and the design of a convergent AFEM algorithm (\cite{HuXu2013,Zhong;Shu;Chen;Xu:2010Convergence,ZhongChenShuWittumEtAl2012}). In this paper, combining the ingredients from both conforming and nonconforming methods, we are able to show that the AFEM algorithm based on the a posteriori error estimation is convergent under common assumptions. Since the nonconforming $P_1$-$P_0$ finite element is element-wisely divergence free, as a result, the a posteriori error estimator only involves the discrete velocity, not the discrete pressure. To the best of our knowledge, this paper is the first work to prove the convergence of an adaptive finite element method of the quad-curl problem. Additionally, in terms of solving the resulting linear systems, the biggest advantage of the decoupled formulation is to allow users taking advantage of the existing fast solvers for Stokes and Maxwell problems.

This paper is organized as follows: Section \ref{sec:fem} introduces the decoupled formulation
as well as its well-posedness. Section \ref{sec:apriori} proves the 
a priori error estimation in both the energy norm and the $\bs L^2$-norm. 
Section \ref{sec:estimator} gives the a posteriori error analysis, Section \ref{sec:quasi-orthogonality} shows the quasi-orthogonality of the solution, and a convergence proof is given in Section \ref{sec:convergence}. In Section \ref{sec:numerics} a comparison of the rates of convergence of the AFEMs using various marking strategies is presented. 

%%%%%%%%%%%%%%%%%%%%%%%%%%%%%%%%%%%%%%%%
\section{A quad-$\curl$ problem and a decoupled formulation}
\label{sec:fem}
%%%%%%%%%%%%%%%%%%%%%%%%%%%%%%%%%%%%%%%%
%Noting that $\curl K_0^c=\curl\bs H_0(\curl, \Omega)$, we get the following exact sequence from the 3D de Rham complex \eqref{eq:deRhamcomplex3d-0}
%\[
%0\; \autorightarrow{}{}K_0^c \autorightarrow{$\curl$}{} \bs H_0(\div , \Omega) \autorightarrow{$\div $}{} L_0^2(\Omega)\autorightarrow{}{} \; 0,
%\]
%which together with Remark 2.15 in \cite{PaulyZulehner2016} implies the exactness of the dual complex
%\[
%0\autorightarrow{}{}L_0^2(\Omega)\;\autorightarrow{$\grad$}{} \; \boldsymbol (\bs H_0(\div , \Omega))' \; \autorightarrow{$\curl$}{} \;(K_0^c)'\;\autorightarrow{}{}\;0.
%\]

Let $\Omega\subset\mathbb R^3$ be a polyhedron homomorphic to a ball, and $\bs f\in \bs H(\div, \Omega)$ with $\div\bs f=0$. Consider the quad-$\curl$ problem
\begin{equation}\label{eq:quadcurl}
\left\{
\begin{aligned}
&(\curl)^4\bs u=\bs f\quad\quad\quad\quad\quad\quad\;\;\text{in }\Omega, \\
&\div\bs u=0\quad\quad\quad\quad\quad\quad\quad\;\;\;\text{in }\Omega, \\
&\bs u\times\bs n=(\curl\bs u)\times\bs n=\bs 0\quad\text{on }\partial\Omega.
\end{aligned}
\right.
\end{equation}
The primal formulation of the quad-$\curl$ problem \eqref{eq:quadcurl} is to find $\bs  u\in\bs H_0(\curlcurl, \Omega)$
such that
\begin{equation}\label{eq:quadcurlprimal}
(\curl\curl\bs u, \curl\curl\bs v)=(\bs f, \bs v) \quad \forall~\bs v\in \bs H_0(\curlcurl, \Omega),
\end{equation}
where by denoting $\Gamma:=\partial \Omega$
\begin{equation}
\label{eq:space-h2curl}
 \begin{aligned}
\bs H_0(\curlcurl, \Omega):=\{&\bs v\in\bs L^2(\Omega, \mathbb{R}^3): \curl \bs v,
\, \curl\curl \bs v\in\bs L^2(\Omega, \mathbb{R}^3), \\
& \div\bs v=0,  \textrm{ and } \bs v\times\bs n=(\curl \bs v)\times\bs n=\boldsymbol{0} \text{ on } \Gamma \}.
\end{aligned} 
\end{equation}

Here we remark that $H_0(\curlcurl, \Omega)$ is same as $H_0(\gradcurl, \Omega)$. This is because  $\curl \bs v\in\bs H_0^1(\Omega, \mathbb{R}^3)$ is equivalent to $\curl \bs v\in\bs H_0(\curl, \Omega)\cap \bs H_0(\div, \Omega)$; see \cite{Girault;Raviart:1986Finite,Zhang2018a}.

A natural mixed method is to mimic the biharmonic equation by introducing $\bs w = \nabla \times \nabla \times \bs u$ and write as a system for which standard edge elements can be used; see \cite{Sun:2016quad-curl}. The main drawback of this decoupling is the loss of the order of convergence due to the fact that boundary condition $(\curl\bs u)\times\bs n=\bs 0$ is imposed weakly. 
Indeed a natural space for $\bs w$ is  $\bs H^{-1}(\curl\curl, \Omega):=\{\bs v\in \bs L^2(\Omega; \mathbb R^3): \curl\curl\bs v\in \bs H^{-1}(\div, \Omega)\}$. Here $\bs H^{-1}(\div, \Omega):=\{\bs v\in\bs H^{-1}(\Omega; \mathbb R^3): \div\bs v\in H^{-1}(\Omega)\}$ is the dual space of $\bs H_0(\curl, \Omega)$ \cite{ChenHuang2018}.
%Note that $\curl\curl\bs v\in \bs H^{-1}(\div, \Omega)$ is equivalent to $\curl\curl\bs v\in \bs H^{-1}(\Omega, \mathbb R^3)$.
Also inheriting from decoupling the biharmonic equation, fast solvers for the linear algebraic system arising from this discretization could be an issue.  

Instead we shall consider a decoupling \cite[Section~3.4]{ChenHuang2018} (see also \cite{Zhang2018a}) so that the optimal order of convergence can be preserved, and meanwhile the solution can be computed efficiently. More importantly for solutions with singularities, the a posteriori error analysis and adaptive finite element methods can be applied to retain optimal order of convergence which is the focus of this work. 

Introduce the space
\begin{equation*}
\label{eq:space-kc0}
K_0^c:=\{\boldsymbol{\phi}\in \bs H_0(\curl, \Omega): \mathrm{div}\boldsymbol{\phi}=0\}=\bs H_0(\curl, \Omega)/\grad H_0^1(\Omega)
\end{equation*}
equipped with norm $\|\cdot\|_{H(\curl)}$.
Due to the following commutative diagram
\begin{equation*}
\begin{array}{c}
\xymatrix@R-=1.0pc{
\bs H_0^{1}(\Omega; \mathbb R^3) \ar[r]^-{\Delta} & \boldsymbol{H}^{-1}(\Omega; \mathbb{R}^3) & \\
L_0^2(\Omega) \ar[r]^-{\grad} & \boldsymbol H^{-1}({\curl}, \Omega) \ar[r]^-{\curl}\ar@{}[u]|{\bigcup}
                & (K_0^c)' \ar[r]^-{} & 0 \\
&\boldsymbol H_0(\div , \Omega) \ar[u]^{I} & K_0^c \ar[u]_{\curl\curl}\ar[l]_-{\curl}}
\end{array},
\end{equation*}
the primal formulation \eqref{eq:quadcurlprimal} of the quad-$\curl$ problem can be decoupled into the following three systems \cite[Section~3.4]{ChenHuang2018} (see also \cite{Zhang2018a}):

\medskip

\noindent {\rm Step 1.}
Given $\bs f\in \bs L^2(\Omega)$, find $\bs w\in \bs H_0(\curl, \Omega)$, $\sigma\in H_0^{1}(\Omega)$ s.t.
\begin{align}
(\curl  \bs w, \curl  \bs v)+(\bs v,\nabla \sigma)&=(\bs f, \bs v)  \quad\quad\quad\quad \forall~\bs v\in \bs H_0(\curl, \Omega),  \label{eq:quarticcurldecouple1}\\
(\bs w,\nabla \tau) &= 0 \quad\quad\quad\quad\quad\;\;\;\, \forall~ \tau\in H_0^{1}(\Omega).  \label{eq:quarticcurldecouple2}
\end{align}

\smallskip

\noindent {\rm Step 2.} Given $\bs w$ computed in Step 1, find $\boldsymbol{\phi}\in \boldsymbol H_0^1(\Omega; \mathbb{R}^3)$, $p\in L_0^{2}(\Omega)$ s.t.
\begin{align}
(\boldsymbol\nabla\boldsymbol\phi, \boldsymbol\nabla\boldsymbol\psi) + (\div\boldsymbol\psi, p) & =(\curl  \bs w, \boldsymbol\psi) \quad\;\;\, \forall~\boldsymbol{\psi}\in  \boldsymbol H_0^1(\Omega; \mathbb{R}^3), \label{eq:quarticcurldecouple3}\\
(\div\boldsymbol\phi, q) &= 0 \quad\quad\quad\quad\quad\;\;\;\, \forall~ q\in L_0^{2}(\Omega).  \label{eq:quarticcurldecouple4}
\end{align}

\smallskip

\noindent {\rm Step 3.} Given $\bs \phi$ computed in Step 2, find 
$\bs u\in \bs H_0(\curl, \Omega)$ and $\xi\in H_0^{1}(\Omega)$ s.t. 
\begin{align}
(\curl \bs  u, \curl \bs  \chi)+(\bs \chi,\nabla \xi)&= (\boldsymbol{\phi}, \curl \bs \chi) \quad\quad \forall~\bs \chi\in \bs H_0(\curl, \Omega), \label{eq:quarticcurldecouple5} \\
(\bs u,\nabla \zeta) &= 0 \quad\quad\quad\quad\quad\;\;\;\, \forall~ \zeta\in H_0^{1}(\Omega).  \label{eq:quarticcurldecouple6}
\end{align}
In other words, the primal formulation \eqref{eq:quadcurlprimal} of the quad-$\curl$ problem \eqref{eq:quadcurl} can be decoupled into two Maxwell equations and one Stokes equation. 

Each system is well-posed and the solution $(\bs w, \sigma, \bs \phi, p, \bs u, \xi)$ to \eqref{eq:quarticcurldecouple1}-\eqref{eq:quarticcurldecouple6} exists and is unique. Now we show briefly, without resorting to the abstract framework in \cite{ChenHuang2018}, the equivalence of the decoupled formulation  \eqref{eq:quarticcurldecouple1}-\eqref{eq:quarticcurldecouple6} and the primary formulation \eqref{eq:quadcurlprimal}. 

By taking $\chi = \nabla \xi$ in \eqref{eq:quarticcurldecouple5}, we conclude the Lagrange multiplier $\xi = 0$. Therefore \eqref{eq:quarticcurldecouple5} becomes $\bs \phi = \curl \bs u$. Notice that the boundary condition $\curl \bs u \times \bs n = 0$ implies that the tangential trace $\bs \phi \times \bs n$ is zero, while $\bs u\times \bs n = 0$ on boundary implies the normal trace $\bs \phi\cdot \bs n = \rot_{\Gamma} \bs u = 0$. 
Together with $\curl \bs \phi = \curl \curl \bs u\in \bs L^2(\Omega)$ and $\div \bs \phi = 0$, by the embedding $\bs H_0(\curl,\Omega)\cap \bs H_0(\div, \Omega)\hookrightarrow \bs H_0^1(\Omega)$ \cite{RaviartThomas1977}, we conclude that $\bs \phi = \curl \bs u\in \bs  H_0^1(\Omega)$.

Furthermore by the identity 
\[
(\boldsymbol\nabla\boldsymbol\phi, \boldsymbol\nabla\boldsymbol\psi)=(\curl\boldsymbol\phi, \curl\boldsymbol\psi)+(\div\boldsymbol\phi, \div\boldsymbol\psi)\quad\forall~\boldsymbol{\phi}, \boldsymbol\psi\in\boldsymbol H_0^1(\Omega; \mathbb{R}^3),
\]
and $\div \bs \phi = 0$,   we can rewrite \eqref{eq:quarticcurldecouple3} as
\begin{equation}\label{eq:curlphi}
(\curl\boldsymbol\phi, \curl\boldsymbol\psi) + (\div\boldsymbol\psi, p)  =(\curl  \bs w, \boldsymbol\psi) \quad\;\;\, \forall~\boldsymbol{\psi}\in  \boldsymbol H_0^1(\Omega; \mathbb{R}^3).
\end{equation}

Noticing the fact $\div\bs f=0$, by choosing $\bs v = \nabla \sigma$, we get from \eqref{eq:quarticcurldecouple1} that the Lagrange multipliers  $\sigma$ is also zero. Now choosing $\bs \psi = \curl \bs v$ in \eqref{eq:curlphi} for a $\bs v \in \bs H_0(\curlcurl, \Omega)$, we get 
$$
(\curl\curl\bs u, \curl\curl\bs v) = (\curl\boldsymbol\phi, \curl\boldsymbol\psi) = (\curl  \bs w, \boldsymbol\psi) = (f, \bs v),
$$
which verifies that the solution $\bs u$ to \eqref{eq:quarticcurldecouple5}-\eqref{eq:quarticcurldecouple6} is also the solution to \eqref{eq:quadcurlprimal} and vice versa.

\begin{remark}\rm
The decoupled formulation \eqref{eq:quarticcurldecouple1}-\eqref{eq:quarticcurldecouple6} also works
for the case $\div\bs f\neq 0$.
By taking $\bs{v} = \nabla \tau$ in \eqref{eq:quarticcurldecouple1}, we get $\Delta\sigma=\div\bs f$.
After deriving $\sigma$, we can simply replace the right hand side $\bs f$ by $\bs f-\nabla\sigma$, which is divergence-free. This is a Helmholtz decomposition where the non-compatible right-hand side's divergence has been taken into account without being explicitly formulated.
\end{remark}
% \begin{remark}\rm 
% Consider $(\curl)^4\bs u+\bs u=\bs f$ with $\div \bs f$ being not necessary to be zero.
% By rewriting it as $(\curl)^4\bs u=\bs f-\bs u$, equation $(\curl)^4\bs u+\bs u=\bs f$ is equivalent to find $\bs w\in \bs H_0(\curl, \Omega)$, $\boldsymbol{\phi}\in \boldsymbol H_0^1(\Omega; \mathbb{R}^3)$, $p\in L_0^{2}(\Omega)$, $\bs u\in \bs H_0(\curl, \Omega)$ and $\xi\in H_0^{1}(\Omega)$ s.t. (cf. \cite[(3.18)]{Zhang2018a})
% \begin{align*}
% (\curl  \bs w, \curl  \bs v)+(\bs w,\nabla \tau)&=(\bs f-\bs u, \bs v)  \quad\quad \forall~\bs v\in \bs H_0(\curl, \Omega), \tau\in H_0^{1}(\Omega), \\
% (\boldsymbol\nabla\boldsymbol\phi, \boldsymbol\nabla\boldsymbol\psi) + (\div\boldsymbol\psi, p) & =(\curl  \bs w, \boldsymbol\psi) \quad\;\;\, \forall~\boldsymbol{\psi}\in  \boldsymbol H_0^1(\Omega; \mathbb{R}^3), \\
% (\div\boldsymbol\phi, q) &= 0 \quad\quad\quad\quad\quad\;\;\;\, \forall~ q\in L_0^{2}(\Omega),  \\
% (\curl \bs  u, \curl \bs  \chi)+(\bs \chi,\nabla \xi)&= (\boldsymbol{\phi}, \curl \bs \chi) \quad\quad \forall~\bs \chi\in \bs H_0(\curl, \Omega).
% \end{align*}
% \end{remark}

\begin{remark}\rm 
As we have shown the auxiliary function $\bs \phi =\curl \bs u$, but $\bs w \neq \curl\curl \bs u$. Equation \eqref{eq:quarticcurldecouple1} can be equivalently written as $\curl\curl \bs w = \bs f$ but now $\bs w \in \bs H_0(\curl,\Omega)$ while $\curl\curl \bs u\in\bs H_0(\div,\Omega)$ may not satisfy the tangential boundary condition. 
\end{remark}

\section{Discrete Methods and A Priori Error Analysis}
\label{sec:apriori}
We consider a conforming mixed finite element method of the Maxwell equations~\eqref{eq:quarticcurldecouple1}-\eqref{eq:quarticcurldecouple2}, and \eqref{eq:quarticcurldecouple5}-\eqref{eq:quarticcurldecouple6} but a nonconforming method for Stokes equation~\eqref{eq:quarticcurldecouple3}-\eqref{eq:quarticcurldecouple4}. We refer to~\cite{Zhang2018a} for a conforming mixed finite element method. 

%One motivation to use non-conforming method 

Let $\{\mathcal T_h\}$ be a family of triangulation of $\Omega$ with mesh size $h=\max_{K\in \mathcal T_h}h_K$, where $h_K$ is the diameter of the tetrahedron $K$.
Denote the $p$-th order Lagrange element space by
\[
V_h^p:=\{v_h\in H_0^{1}(\Omega): v_h|_K\in\mathbb P_p(K) \textrm{ for each } K\in\mathcal T_h\},
\]
and the lowest-order N\'ed\'elec edge element space \cite{Nedelec1980} by
\[
\bs V_h^c:=\{\bs v_h\in \bs H_0(\curl, \Omega): \bs v_h|_K\in\mathbb P_{0}(K;\mathbb R^3)\oplus\bs x\wedge\mathbb P_{0}(K;\mathbb R^3) \textrm{ for each } K\in\mathcal T_h\}.
\]
We use $V_h^1-\bs V_h^c$ to discretize the Maxwell equation~\eqref{eq:quarticcurldecouple1}-\eqref{eq:quarticcurldecouple2}. 
Find $\bs w_h\in \bs V_h^c$, $\sigma_h\in V_h^1$ s.t.
\begin{align}
(\curl \bs w_h, \curl \bs v_h)+(\bs v_h,\nabla \sigma_h)&=(\bs f, \bs v_h)  \quad\quad\quad\quad\;\; \forall~\bs v_h\in \bs V_h^c,  \label{eq:mfem1}\\
(\bs w_h,\nabla \tau_h) &= 0 \quad\quad\quad\qquad\qquad \forall~ \tau_h\in V_h^1.  \label{eq:mfem2}
\end{align}

We then use the nonconforming $P_1$-$P_0$ element \cite{CrouzeixRaviart1973} to discretize the Stokes problem \eqref{eq:quarticcurldecouple3}-\eqref{eq:quarticcurldecouple4}. To this end, let
\begin{equation*}
%\label{eq:space-cr}
\begin{aligned}
\bs V_h^{\rm CR} :=
\{
\bs\psi_h\in \boldsymbol L^2(\Omega; \mathbb{R}^3): &\,\bs \psi_h|_K\in\mathbb P_{1}(K;\mathbb R^3) \textrm{ for each } K\in\mathcal T_h, \\
&\textrm{ and } (\llbracket \bs \psi_h\rrbracket, 1)_F=0  \textrm{ for each } F\in\mathcal F_h
\},
\end{aligned} 
\end{equation*}
where $\llbracket \bs v \rrbracket (\bs x):= \lim\limits_{\epsilon\to 0^+} \left( \bs v|_{K_1} (\bs x-\epsilon \bs n_{K_1}) -  \bs v|_{K_2}(\bs x+\epsilon \bs n_{K_1}) \right) $ is defined as the jump on face $F$ for $\bs x \in F$ and $\boldsymbol n_{K_1}$ being the outer unit normal to $K_1$ on face $F$. Denote the piecewise constant space as 
\[
\mathcal Q_h :=\{q_h\in L_0^2(\Omega): \; q_h|_K\in\mathbb P_0(K) \textrm{ for each } K\in\mathcal T_h\}.
\]
Given $\bs w_h$ computed from \eqref{eq:mfem1}-\eqref{eq:mfem2}, find $\boldsymbol{\phi}_h\in  \bs V_h^{\rm CR}$, $p_h\in \mathcal Q_h$ s.t.
\begin{align}
(\boldsymbol\nabla_h\boldsymbol\phi_h, \boldsymbol\nabla_h\boldsymbol\psi_h) + (\div_h\boldsymbol\psi_h, p_h) & =(\curl\bs w_h, \boldsymbol\psi_h) \quad\;\;\; \forall~\boldsymbol{\psi}_h\in  \bs V_h^{\rm CR}, \label{eq:mfem3}\\
(\div_h\boldsymbol\phi_h, q_h) &= 0 \quad\quad\quad\quad\qquad\quad \forall~ q_h\in \mathcal Q_h. \label{eq:mfem4}
\end{align}
Hereafter $\boldsymbol\nabla_h$, $\curl_h$ and $\div_h$ mean the element-wise defined counterparts of 
$\boldsymbol\nabla$, $\curl$ and $\div$ with respect to $\mathcal T_h$.

%\LC{Adding description on the linear edge element space here.}
Upon solving the system above, $\bs \phi_h$ is a second-order approximation to $\curl \bs u$ when the data is smooth. Finally, when one needs to seek a better approximation to $\bs u$ under $\bs L^2$-norm, $\bs u_h\in \bs V_h^{c_1}$ and $\xi_h\in V_h^{2}$ are sought such that
they satisfy
\begin{align}
(\curl \bs  u_h, \curl \bs  \chi_h)+(\bs \chi_h,\nabla \xi_h)&= (\boldsymbol{\phi}_h, \curl \bs \chi_h) \quad\quad \forall~\bs \chi_h\in \bs V_h^{c_1}, \label{eq:mfem5} \\
(\bs u_h,\nabla \zeta_h) &= 0 \quad\quad\quad\quad\quad\qquad \forall~ \zeta_h\in V_h^{2}.  \label{eq:mfem6}
\end{align}
Here $\bs V_h^{c_1}$ is the linear second family of N\'{e}d\'{e}lec element:
\[
\bs V_h^{c_1}:=\{\bs v_h\in \bs H_0(\curl, \Omega): \bs v_h|_K\in\mathbb P_{1}(K;\mathbb R^3) \textrm{ for each } K\in\mathcal T_h\},
\]
and $V_h^{2}$ is the quadratic Lagrange element.

The finite element pair $(\bs V_h^{\rm CR}, \mathcal Q_h)$ is stable for the Stokes equation \cite{BoffiBrezziFortin2013}, i.e., we have for any $\widetilde{\bs\phi}_h\in \bs V_h^{\rm CR}$ and $\widetilde{p}_h\in\mathcal Q_h$ that
\begin{equation*}%\label{eq:infsupncfmP1P0}
|\widetilde{\bs\phi}_h|_{1,h}+
\|\widetilde p_h\|_0\lesssim \sup_{\substack{ \bs\psi_h\in \bs V_h^{\rm CR}\\ q_h\in\mathcal Q_h} }
\frac{(\boldsymbol\nabla_h\widetilde{\boldsymbol\phi}_h, \boldsymbol\nabla_h\boldsymbol\psi_h) 
+ 
(\div_h\boldsymbol\psi_h, \widetilde{p}_h)+
(\div_h\widetilde{\boldsymbol\phi}_h, q_h)}{|\bs\psi_h|_{1,h}+\|q_h\|_0}.
\end{equation*}

\subsection{A priori error analysis}
Next we focus on the a priori error analysis for the decoupled mixed finite element method \eqref{eq:mfem1}-\eqref{eq:mfem6}. First of all, since $\div\bs f=0$ and $\nabla V_h^1\subset \bs V_h^c$, we get from \eqref{eq:mfem1} and \eqref{eq:mfem5} that $\sigma_h=0$ and $\xi_h=0$.

\begin{lemma}[Galerkin orthogonality]
Let $(\bs w, 0)\in \bs H_0(\curl, \Omega)\times H_0^{1}(\Omega)$ be the solution of the Maxwell equation \eqref{eq:quarticcurldecouple1}-\eqref{eq:quarticcurldecouple2}, and $(\bs w_h, 0)\in \bs V_h^c\times V_h^1$ be the solution of the mixed method \eqref{eq:mfem1}-\eqref{eq:mfem2}. Then
\begin{equation}\label{eq:GalerkinOrthcurl}
(\curl (\bs w-\bs w_h), \curl \bs v_h)=0\quad \forall~\bs v_h\in \bs V_h^c.
\end{equation}
\end{lemma}
\begin{proof}
As the Lagrange multiplier $\sigma = 0$ and its approximation $\sigma_h = 0$, subtracting \eqref{eq:mfem1} from \eqref{eq:quarticcurldecouple1}, we get the desired orthogonality.
\end{proof}
The error analysis of the mixed finite element method \eqref{eq:mfem1}-\eqref{eq:mfem2} 
% can be found in \cite{Arnold2018,ArnoldFalkWinther2006} 
is first studied by F. Kikuchi in \cite{Kikuchi:1989mixed,Kikuchi:1989discrete}.
We recall it for completeness.
\begin{lemma}
\label{lem:estimate-w}
Let $(\bs w, 0)\in \bs H_0(\curl, \Omega)\times H_0^{1}(\Omega)$ be the solution of the Maxwell equation \eqref{eq:quarticcurldecouple1}-\eqref{eq:quarticcurldecouple2}, and $(\bs w_h, 0)\in \bs V_h^c\times V_h^1$ the solution of the mixed method \eqref{eq:mfem1}-\eqref{eq:mfem2}.
Assume $\curl\bs w \in \boldsymbol H^1(\Omega; \mathbb{R}^3)$, then we have
\begin{equation}\label{eq:errorpriorwh}
\|\curl (\bs w-\bs w_h)\|_0\lesssim h|\curl\bs w|_1.
\end{equation}
\end{lemma}
\begin{proof}
The orthogonality \eqref{eq:GalerkinOrthcurl} implies the best approximation
%Then we have for any $\bs v_h\in \bs V_h^c$ that
%\begin{align*}
%\|\curl (\bs w-\bs w_h)\|_0^2&=(\curl (\bs w-\bs w_h), \curl (\bs w-\bs w_h))=(\curl (\bs w-\bs w_h), \curl (\bs w-\bs v_h)) \\
%&\leq \|\curl (\bs w-\bs w_h)\|_0\|\curl (\bs w-\bs v_h)\|_0,
%\end{align*}
%which implies
\[
\|\curl (\bs w-\bs w_h)\|_0\leq \inf_{\bs v_h\in \bs V_h^c} \|\curl (\bs w-\bs v_h)\|_0.
\]
This gives \eqref{eq:errorpriorwh} by an interpolation error estimate (see e.g., \cite{Monk2003}).
\end{proof}

According to the Poincar\'e-Friedrichs inequality for piecewise $H^1$ functions \cite{Brenner2003}, the following inequality holds
\begin{equation*}%\label{eq:poincareVhs}
\|\psi_h\|_0\lesssim |\psi_h|_{1,h}\quad\forall~\psi_h\in \bs V_h^{\rm CR}+\boldsymbol H_0^1(\Omega; \mathbb{R}^3).
\end{equation*}
Denote $\bs I_h^{s}$ as the nodal interpolation operator from $\boldsymbol H_0^1(\Omega; \mathbb{R}^3)$ to $\bs V_h^{\rm CR}$, then
\begin{equation}\label{eq:errorestimateIhs1}
(\bs\nabla(\bs\psi-\bs I_h^{s} \bs\psi), \bs\tau)_K=0\quad\forall~\bs\psi\in\boldsymbol H^1(\Omega; \mathbb{R}^3), \bs\tau\in\mathbb P_0(K;\mathbb M), \; \; K\in\mathcal T_h,
\end{equation}
where $P_0(K;\mathbb M)$ stands for the space of constant $3\times 3$ matrix on $K$, and for $j=1,2$,
%\[
%\div_h(\bs I_h^{s}\bs\psi)= Q_h(\div\bs\psi)\quad\forall~\bs\psi\in\boldsymbol H^1(\Omega; \mathbb{R}^3),
%\]
\begin{equation}\label{eq:errorestimateIhs}
\|\bs\psi-\bs I_h^{s}\bs\psi\|_{0,K}+h_K|\bs\psi-\bs I_h^{s}\bs\psi|_{1,K}\lesssim h_K^j|\bs\psi|_{j,K}\quad \forall~\bs\psi\in\boldsymbol H^j(\Omega; \mathbb{R}^3), \;\; K\in\mathcal T_h.
\end{equation}

The error analysis for the nonconforming $P_1$-$P_0$ element approximation \eqref{eq:mfem3}--\eqref{eq:mfem4} of Stokes equation is standard \cite{CrouzeixRaviart1973}. Using the decoupled system to approximate the quad--curl problem, the subtlety is the perturbation of data. We shall present a stability result for using $\curl \bs w_h$ to approximate $\curl \bs w$. To this end, we introduce the space 
$$
\bs Z := \boldsymbol{H}_0^1(\Omega; \mathbb R^3)\cap \ker(\div).
$$ 
Subsequently \eqref{eq:quarticcurldecouple3} and the continuous problem using the perturbed data can be written as follows: 
\[
- \Delta \bs \phi = \curl \bs w \;\text{ in }\; \bs Z'
\quad \text{and} \quad -\Delta \tilde{\bs \phi} = \curl \bs w_h \;\text{ in } \bs Z',
\]
respectively. The second problem above is equivalent to 
\begin{align}
(\boldsymbol\nabla\tilde{\bs \phi}, \boldsymbol\nabla\boldsymbol\psi) + (\div\boldsymbol\psi, \tilde{p}) & =(\curl  \bs w_h, \boldsymbol\psi) \quad\;\;\, \forall~\boldsymbol{\psi}\in  \boldsymbol H_0^1(\Omega; \mathbb{R}^3), 
\label{eq:quarticcurldecouple-perturbed1}
\\
(\div\tilde{\bs \phi}, q) &= 0 \quad\quad\quad\quad\quad\;\;\;\;\; \forall~ q\in L_0^{2}(\Omega). 
\label{eq:quarticcurldecouple-perturbed2}
\end{align}
The analysis is performed for this problem with the perturbed data.

\begin{lemma}
\label{lem:estimate-phitilde}
 Let $(\bs\phi, p), (\tilde{\bs \phi}, \tilde p)$ be the solutions to \eqref{eq:quarticcurldecouple3}--\eqref{eq:quarticcurldecouple4} and \eqref{eq:quarticcurldecouple-perturbed1}--\eqref{eq:quarticcurldecouple-perturbed2}, respectively, where $\bs w$ and $\bs w_h$ satisfy the orthogonality \eqref{eq:GalerkinOrthcurl}. Then 
%$$\| \tilde{\bs \phi} - \bs \phi\|_1 + \| \tilde p - p\|_0 \lesssim h \| \curl (\bs w - \bs w_h)\|_0.$$
$$\| \tilde{\bs \phi} - \bs \phi\|_1 \lesssim h \| \curl (\bs w - \bs w_h)\|_0.$$
\end{lemma}
\begin{proof}
The difference between the two pairs satisfies the Stokes equation
\begin{align*}
-\Delta (\bs \phi - \tilde{\bs \phi}) + \nabla (p - \tilde p) &= \curl (\bs w - \bs w_h) \quad\text{ in } (\bs H_0^1(\Omega; \mathbb R^3))',\\
\div (\bs \phi - \tilde{\bs \phi}) &= 0 \qquad\qquad\qquad\;\, \text{ in } L_0^{2}(\Omega).
\end{align*}
Applying the definition of the duality pair testing against 
$\bs \phi - \tilde{\bs \phi}$, we get
\[
|\bs \phi - \tilde{\bs \phi}|_1^2=(\curl (\bs w - \bs w_h), \bs \phi - \tilde{\bs \phi}).
\]
Moreover, since $\div(\bs \phi - \tilde{\bs \phi})=0$, by \cite[Chapter 1 Theorem 3.4]{Girault;Raviart:1986Finite} there exists $\bs v\in \bs H_0^1(\Omega; \mathbb R^3)$ such that
\[
\bs \phi - \tilde{\bs \phi}=\curl\bs v,\quad \|\bs v\|_1\lesssim \|\bs \phi - \tilde{\bs \phi}\|_0.
\]
Then it follows from \eqref{eq:GalerkinOrthcurl} that
\[
|\bs \phi - \tilde{\bs \phi}|_1^2=(\curl (\bs w - \bs w_h), \curl\bs v)=(\curl (\bs w - \bs w_h), \curl(\bs v-\bs v_h)), \quad\forall~\bs v_h\in \bs V_h^c.
\]
Thus
\[
|\bs \phi - \tilde{\bs \phi}|_1^2\leq \|\curl (\bs w - \bs w_h)\|_0\inf_{\bs v_h\in \bs V_h^c}\|\curl (\bs v - \bs v_h)\|_0\lesssim h\|\curl (\bs w - \bs w_h)\|_0|\bs \phi - \tilde{\bs \phi}|_1,
\]
which implies the desired result.
\end{proof}
%\LC{Do we need the estimate of $p$?}

In the next step, we treat $\bs \phi_h$ as the approximation of $\tilde{\bs \phi}$ and use the standard error analysis to obtain the following estimate. Here the $H^2$-regularity of Stokes equation is assumed to hold.  

\smallskip
\begin{itemize}[leftmargin=2.5em]
  \item[(H2)] \label{assumption:h2} Given an $\bs f\in \boldsymbol{L}^2(\Omega; \mathbb R^3)$, let $\bs u\in \bs H_0^1(\Omega; \mathbb R^3)$ and $p \in L^2_0(\Omega)$ be the solution of the Stokes equation 
$$
-\Delta\boldsymbol{u} + \nabla p = \bs f, \quad  \div \bs u = 0. 
$$
Then $\bs u\in \bs H^2(\Omega; \mathbb R^3)$ and $p\in H^1(\Omega)$ and 
$$
\|\bs u\|_2 +\|p\|_1\lesssim \|f\|_0.
$$
\end{itemize}

It is well known that the assumption \hyperref[assumption:h2]{(H2)} holds for smooth or convex domain $\Omega$ (e.g., see \cite[Section 11.5]{Mazya;Rossmann:2010Elliptic}). In particular, assuming \hyperref[assumption:h2]{(H2)} holds, we have
$$
\|\bs \phi\|_2+\|p\|_1\lesssim\|\curl \bs w\|_0,
\quad \|\tilde{\bs \phi}\|_2+\|\tilde{p}\|_1\lesssim\|\curl \bs w_h\|_0,
$$
thus the standard a priori estimate for the stable nonconforming $P_1$-$P_0$ pair holds.

% \LC{State out the regularity result of Stokes equation.}
% \LC{Recall the result below and cite a reference } 
%and a continuity argument is used to bound $\|\tilde{\bs\phi}\|_2$ by $\|\curl \bs w\|_0$ 
%\mnote{\ Here by the a stability estimate we only have $\|\tilde{\bs\phi}\|_2\lesssim 
%\|\curl \bs w_h\|_0$ on convex or smooth $\Omega$?}. 

\begin{theorem}
\label{theorem:errorpriorphih}
Let $(\bs \phi, p)\in \boldsymbol H_0^1(\Omega; \mathbb{R}^3)\times L_0^{2}(\Omega)$ be the solution of the Stokes equation \eqref{eq:quarticcurldecouple3}-\eqref{eq:quarticcurldecouple4}, and $(\bs \phi_h, p_h)\in \bs V_h^{\rm CR}\times \mathcal Q_h$ the solution of the mixed method \eqref{eq:mfem3}-\eqref{eq:mfem4}.
Assume the $H^2$-regularity of Stokes equation, i.e., \hyperref[assumption:h2]{(H2)} holds, then
\begin{equation}\label{eq:errorpriorphih}
%|\bs \phi-\bs \phi_h|_{1,h}+\|p-p_h\|_0\lesssim h \|\curl\bs w \|_0.
|\bs \phi-\bs \phi_h|_{1,h}\lesssim h \|\curl\bs w \|_0.
\end{equation}
\end{theorem} 
\begin{proof}
First by a standard estimate \cite{CrouzeixRaviart1973}, and the elliptic regularity estimate of the approximation $\bs \phi_h$ for the Stokes problem with $\curl \bs w_h$ as data, we have
\[
|\widetilde{\bs\phi} -\bs \phi_h|_{1,h} \lesssim 
h  |\widetilde{\bs\phi}|_{2} \lesssim h \|\curl {\bs w}_h \|_0.
\]
%Furthermore, by an energy estimate of problem \eqref{eq:mfem1}--\eqref{eq:mfem2}, and a Poincar\'{e}-type inequality for $\bs w_h$ which discrete divergence free
%\[
%\|\curl {\bs w}_h \|_0^2 = (\bs f,\bs w_h)\leq \|\bs f\|_0 \, \|\bs w_h\|_0\leq
% \|\bs f\|_0 \, \|\curl \bs w_h\|_0,
%\]
%thus combining the estimate above with those in Lemma \ref{lem:estimate-phitilde} and \ref{lem:estimate-w} yields the result.
%\[
%|\bs \phi-\bs \phi_h|_{1,h}+\|p-p_h\|_0\lesssim h \left(\|\curl\bs w\|_0 + \|\bs f\|_0 \right)
%\]
%\SC{The estimate I can come up with has an extra term than Long's.}
Furthermore, as ${\bs w}_h$ is the projection of $\bs w$ to the discrete space in the energy norm, from the orthogonality \eqref{eq:GalerkinOrthcurl}, we have $\|\curl {\bs w}_h \|_0 \leq \|\curl {\bs w} \|_0$. Consequently, the theorem follows from combining the estimate with the ones in Lemma \ref{lem:estimate-phitilde}.
\end{proof}

\begin{remark}[Nonhomogeneous boundary conditions]
\label{rmk:nonhomogeneous-bc}
By a simple density argument we can see that $\curl \bs u\cdot \bs n = \div_{\Gamma} (\bs u\times \bs n)$. Consequently, the presence of nonhomogeneous $\bs u\times\bs n$ and/or $\curl \bs u\times \bs n$ leads to the necessity of impose compatible Dirichlet boundary conditions with the divergence free condition for problems \eqref{eq:quarticcurldecouple3}--\eqref{eq:quarticcurldecouple4}, \eqref{eq:quarticcurldecouple-perturbed1}--\eqref{eq:quarticcurldecouple-perturbed2}. Let $\bs \phi_I$ be the standard nodal interpolation in Crouzeix-Raviart element of a sufficiently smooth $\bs\phi$, by a standard decomposition argument we can see that aside from the terms on the right hand side of \eqref{eq:errorpriorphih}, for the nonhomogeneous boundary condition, the estimate should include:
\[
| \bs\phi - \bs \phi_I|_{1/2,h, \partial \Omega}:= \left(\sum_{F\in \mathcal{F}_h}
| \bs\phi - \bs \phi_I|_{1/2,F}^2\right)^{1/2}
 \lesssim h |\bs\phi|_2.
\]
%\LC{ if the regularity holds, $ |\bs\phi|_2\lesssim \|\curl w\|$}
%\SC{$\bs w$ may be zero while $\bs \phi$ is not. Consider a nontrivial biharmonic potential $\mu$ such that $\Delta^2 \mu=0$, and $\bs u = (0,0,\mu)$, then $\bs f = (0,0,\Delta^2 \mu)$ implies $\bs w=0$.}.

\end{remark}

Next we present the $L^2$-error estimate for the Stokes equation. 
\begin{lemma}
Let $(\bs w_h, 0, \bs \phi_h, p_h)\in \bs V_h^c\times V_h^1\times\bs V_h^{\rm CR}\times \mathcal Q_h$ be the solution of the mixed method \eqref{eq:mfem1}-\eqref{eq:mfem4} on triangulation $\mathcal T_h$. Assume $H^2$-regularity of Stokes equation holds, i.e., \hyperref[assumption:h2]{(H2)} holds, then 
%We have
\begin{equation}\label{eq:phihL2error}
\|\bs\phi-\bs\phi_h\|_0\lesssim h|\bs\phi-\bs\phi_h|_{1,h}+h\|\curl(\bs w-\bs w_h)\|_0 + h^2\|\curl\bs w\|_0.
\end{equation}
Furthermore if $\curl\bs w\in \bs H^1(\Omega; \mathbb R^3)$, then we have the second order estimate
\begin{equation}
\label{eq:phihL2error-rate}
\|\bs\phi-\bs\phi_h\|_0\lesssim h^2 \| \curl \bs w \|_1.
\end{equation} 
\end{lemma}
\begin{proof}
Consider the following dual problem: seek $(\hat{\bs \phi},\hat p)\in \bs H_0^1(\Omega; \mathbb R^3)\times L^2_0(\Omega)$ such that
\[
\left\{
\begin{aligned}
-\Delta \hat{\bs \phi} + \nabla \hat p &= \bs\phi-\bs\phi_h,\\
\div \hat{\bs \phi} &= 0.
\end{aligned}
\right.
\]
The $H^2$-regularity to the problem above (e.g., see \cite[Section 11.5]{Mazya;Rossmann:2010Elliptic}) reads
\begin{equation}\label{eq:dualstokesregularity}
\|\hat{\bs \phi}\|_2+\|\hat p\|_1\lesssim \|\bs\phi-\bs\phi_h\|_0.
\end{equation}
Since $\div_h(\bs\phi-\bs\phi_h)=0$, it follows
\begin{align}
\|\bs\phi-\bs\phi_h\|_0^2&=(\bs\phi-\bs\phi_h, -\Delta \hat{\bs \phi} + \nabla \hat p) \notag\\
&=(\nabla_h(\bs\phi-\bs\phi_h), \nabla\hat{\bs \phi}) + \sum_{K\in\mathcal T_h}(\bs\phi-\bs\phi_h, \hat p\bs n-\partial_n\hat{\bs \phi})_{\partial K}. \label{eq:20200918-1}
\end{align}
Employing \eqref{eq:errorestimateIhs1} and the fact $\div_h\bs I_h^{s}\hat{\bs \phi}=0$, we obtain
\begin{align*}
(\nabla_h(\bs\phi-\bs\phi_h), \nabla\hat{\bs \phi})&=(\nabla\bs\phi, \nabla\hat{\bs \phi})-(\nabla_h\bs\phi_h, \nabla_h\bs I_h^{s}\hat{\bs \phi}) =(\curl\bs w, \hat{\bs \phi})-(\curl\bs w_h, \bs I_h^{s}\hat{\bs \phi}) \\
&=(\curl\bs w-\curl\bs w_h, \hat{\bs \phi})+(\curl\bs w_h, \hat{\bs \phi}-\bs I_h^{s}\hat{\bs \phi}).
\end{align*}
Applying the same argument in Lemma~\ref{lem:estimate-phitilde} by treating $\hat{\bs \phi}$ as a stream function and inserting a curl of its interpolation,  we achieve
\[
(\curl\bs w-\curl\bs w_h, \hat{\bs \phi})\lesssim h\|\curl(\bs w-\bs w_h)\|_0|\hat{\bs \phi}|_1.
\]
Besides from \eqref{eq:errorestimateIhs}, we have 
\[
(\curl\bs w_h, \hat{\bs \phi}-\bs I_h^{s}\hat{\bs \phi})\lesssim h^2\|\curl\bs w\|_0|\hat{\bs \phi}|_2.
\]
Hence 
\begin{equation}\label{eq:20200918-2}
(\nabla_h(\bs\phi-\bs\phi_h), \nabla\hat{\bs \phi})\lesssim h\|\curl(\bs w-\bs w_h)\|_0|\hat{\bs \phi}|_1 + h^2\|\curl\bs w\|_0|\hat{\bs \phi}|_2.
\end{equation}
Due to the continuity condition of Crouzeix-Raviart element, %in \eqref{eq:space-cr}, 
by a standard technique of inserting a constant on each face (e.g., see \cite[Chapter 10.3]{Brenner;Scott:2008mathematical}) we get
\begin{equation}\label{eq:20200918-3}
\sum_{K\in\mathcal T_h}(\bs\phi-\bs\phi_h, \hat p\bs n-\partial_n\hat{\bs \phi})_{\partial K}\lesssim h|\bs\phi-\bs\phi_h|_{1,h}(\|\hat{\bs \phi}\|_2+\|\hat p\|_1).
\end{equation}
Combining \eqref{eq:20200918-1}-\eqref{eq:20200918-3} and \eqref{eq:dualstokesregularity} yields
\[
\|\bs\phi-\bs\phi_h\|_0\lesssim h|\bs\phi-\bs\phi_h|_{1,h}
+h\|\curl(\bs w-\bs w_h)\|_0 + h^2\|\curl\bs w\|_0,
\]
which is \eqref{eq:phihL2error}.
%
%Duality argument for  $\|\tilde{\bs\phi}-\bs\phi_h\|_0$ is standard. Then using the triangle inequality and the stability result.
%
%\LC{Can we remove $p - p_h$? Consider the duality argument in $Z$ and $Z_h$.}
\end{proof}

We now consider the approximation \eqref{eq:mfem5}--\eqref{eq:mfem6} of the last Maxwell equation. Due to the inexactness of the data, the orthogonality is lost, but the perturbation is measured in $L^2$-norm of the difference $ \boldsymbol{\phi}-\boldsymbol{\phi}_h$, which is controllable.

\begin{lemma}
\label{lem:convergence-curlu}
Let $(\bs u, 0)\in \bs H_0(\curl, \Omega)\times H_0^{1}(\Omega)$ be the solution of the Maxwell equation \eqref{eq:quarticcurldecouple5}-\eqref{eq:quarticcurldecouple6}, and $(\bs u_h, 0)\in \bs V_h^{c_1}\times V_h^{2}$ the solution of the mixed method \eqref{eq:mfem5}-\eqref{eq:mfem6}, then
%Assume $\curl\bs u, \curl\bs w  \in \boldsymbol H^1(\Omega; \mathbb{R}^3)$, $\bs \phi \in \boldsymbol H^2(\Omega; \mathbb{R}^3)$ and $p\in H^{1}(\Omega)$, then we have
\begin{equation}
\label{eq:errorprioruh}
\|\curl (\bs u-\bs u_h)\|_0\lesssim \|\boldsymbol{\phi}-\boldsymbol{\phi}_h\|_0+\inf_{\bs v_h\in \bs V_h^c} \|\curl (\bs u-\bs v_h)\|_0.
%h(|\curl\bs u|_1+|\curl\bs w|_1+|\boldsymbol\phi|_2+|p|_1).
\end{equation}
Assume that the $H^2$-regularity of Stokes equation \hyperref[assumption:h2]{(H2)} holds
and $\curl\bs w, \curl\bs u\in \bs H^1(\Omega; \mathbb R^3)$, then
\begin{equation*}
\|\curl (\bs u-\bs u_h)\|_0\lesssim h^2 | \curl \bs w |_1 + h | \curl \bs u |_1.
\end{equation*}
\end{lemma}
\begin{proof}
Subtracting \eqref{eq:mfem5} from \eqref{eq:quarticcurldecouple5}, we get
\begin{equation}
\label{eq:orthogonality-uuh}
(\curl (\bs u-\bs u_h), \curl \bs \chi_h)=(\boldsymbol{\phi}-\boldsymbol{\phi}_h, \curl \bs \chi_h)\quad \forall~\bs \chi_h\in \bs V_h^{c_1}.
\end{equation}
Taking $\bs \chi_h=\bs v_h- \bs u_h$ with $\bs v_h\in \bs V_h^{c_1}$, we acquire
\begin{align*}
\|\curl (\bs u-\bs u_h)\|_0^2&=(\curl (\bs u-\bs u_h), \curl (\bs u-\bs u_h)) \\
&=(\curl (\bs u-\bs u_h), \curl (\bs u-\bs v_h)) + (\boldsymbol{\phi}-\boldsymbol{\phi}_h, \curl(\bs v_h- \bs u_h)) \\
&\leq \|\curl (\bs u-\bs u_h)\|_0\|\curl (\bs u-\bs v_h)\|_0 \\
&\quad + \|\boldsymbol{\phi}-\boldsymbol{\phi}_h\|_0(\|\curl(\bs u- \bs v_h)\|_0+\|\curl(\bs u- \bs u_h)\|_0),
\end{align*}
which indicates
\[
\|\curl (\bs u-\bs u_h)\|_0\lesssim \|\boldsymbol{\phi}-\boldsymbol{\phi}_h\|_0
+\inf_{\bs v_h\in \bs V_h^{c_1}} \|\curl (\bs u-\bs v_h)\|_0.
\]
%Finally we achieve \eqref{eq:errorprioruh} from \eqref{eq:errorpriorphih} and the Poincar\'e inequality \eqref{eq:poincareVhs}.
\end{proof}

Recall that $\bs \phi = \curl \bs u$ and $\|\bs \phi - \bs \phi_h\|_0$ is at least first order $h$. Therefore if still merely the lowest order edge element is used in \eqref{eq:mfem5}--\eqref{eq:mfem6}, no approximation to $\curl \bs u$ better than $\bs \phi_h$ could be obtained. By the duality argument for Stokes equation, the error $\|\boldsymbol{\phi} - \boldsymbol{\phi}_h\|_0$ can be of second order $h^2$ if the $H^2$-regularity result holds. 
As a result in the last Maxwell equation, we opt to use the second family N\'ed\'elec element to improve the $L^2$ approximation of $\bs u$ to the second order.

% \LC{Present error analysis for $\| u - u_h \|$.}
\begin{theorem}
\label{theorem:convergence-u}
Let $(\bs u, 0)\in \bs H_0(\curl, \Omega)\times H_0^{1}(\Omega)$ be the solution of the Maxwell equation \eqref{eq:quarticcurldecouple5}-\eqref{eq:quarticcurldecouple6}, and $(\bs u_h, 0)\in \bs V_h^{c_1}\times V_h^{2}$ the solution of the mixed method \eqref{eq:mfem5}-\eqref{eq:mfem6}. Assume $\Omega$ is convex, then 
\begin{equation*}
  %  \| \bs u-\bs u_h \| \lesssim h\|\curl (\bs u-\bs u_h)\| + \|\boldsymbol{\phi}-\boldsymbol{\phi}_h\|.
\begin{aligned}
\|\bs u - \bs u_h\|_0 \lesssim
& \inf_{\bs v_h\in \bs V_h^{c_1}} \Big\{ \|\bs u - \bs v_h\|_0 + h\|\curl (\bs u - \bs v_h)\|_0 \Big\}
\\
& + h \|\curl (\bs u - \bs u_h)\|_0 + \|\bs \phi - \bs\phi_h\|_0,
\end{aligned}
\end{equation*}
and when $\curl\bs w\in \bs H^1(\Omega; \mathbb R^3)$ and $\bs u\in \bs H^2(\Omega; \mathbb R^3)$,
\begin{equation*}
\|\bs u-\bs u_h\|_0\lesssim h^2 (| \curl \bs w |_1 + \| \curl \bs u \|_1 + |\bs u|_2).
\end{equation*}

\end{theorem}
% \LC{Will type later.}
% \mnote{I think there should be a term $\inf_{\bs v_h\in \bs V_h^c} \|\bs u-\bs v_h\|_0$, if $\bs u$ is a pure gradient field then the right hand side is zero.
% }

\begin{proof}
The proof is adapted from a similar argument in \cite{Zhong;Shu;Wittum;Xu:2009Optimal} without the data perturbation. Denote $\bs{e}_h: = \bs u-\bs u_h$, then by \eqref{eq:quarticcurldecouple6} and \eqref{eq:mfem6}, we have $(\bs{e}_h, \nabla \zeta_h) = 0$ for $\zeta_h\in V_h^{2}$, thus for any fixed
$\bs{v}_h \in \bs V_h^{c_1}$
\[
\|\bs e_h\|_0^2 = (\bs e_h, \bs u-\bs v_h) + (\bs e_h, \bs s_h + \nabla q_h) 
= (\bs e_h, \bs u-\bs v_h) + (\bs e_h, \bs s_h),
\]
where a discrete Helmholtz decomposition
\begin{equation}
\label{eq:decomposition-discrete}
\bs v_h -\bs u_h = \bs s_h + \nabla q_h, \quad \text{and } (\bs s_h,\nabla r_h)=0, \quad \forall~r_h\in V_h^{2}
\end{equation}
is applied such that $\bs s_h\in \bs V_h^{c_1}$. As a result,
\begin{equation}
\label{eq:estimateL2-uuh}
\|\bs e_h\|_0 \lesssim \|\bs u-\bs v_h\|_0 + \| \bs s_h \|_0.
\end{equation}
An $\bs H(\curl)$-lifting $\bs s\in \bs H_0(\curl, \Omega)$ (see \cite[Lemma 7.6, Remark 3.52]{Monk2003}) of $\bs s_h$ is sought such that
\[
\curl \bs s = \curl \bs s_h, \quad \div \bs s = 0, \quad \text{ and }\;\;\;\; \|\bs s - \bs s_h\|_0\lesssim h \|\curl \bs s_h\|_0.
\] 
% with a $\delta \in [0,1/2]$ where $\delta = 1/2$ if $\Omega$ is convex or $C^{1,1}$-smooth. 
By the triangle inequality and \eqref{eq:decomposition-discrete}, 
\[
\|\bs s_h\|_0\leq \|\bs s\|_0 + \|\bs s - \bs s_h\|_0 \lesssim \|\bs s\|_0 + h \|\curl \bs s_h\|_0 = \|\bs s\|_0 + h \|\curl (\bs u_h - \bs v_h) \|_0,
\]
hence it suffices to bound $\|\bs s\|_0 $. Consequently, the Aubin-Nitsche argument is applied on $\bs s$, where we seek an  
$(\bs r, \xi) \in \bs H_0(\curl, \Omega)\times H_0^{1}(\Omega)$ s.t. 
\begin{align}
(\curl \bs  r, \curl \bs  \chi)+(\bs \chi,\nabla \xi)&= (\bs s, \bs \chi) \quad\quad \forall~\bs \chi\in \bs H_0(\curl, \Omega), 
\label{eq:auxpb1} \\
(\bs r,\nabla \zeta) &= 0 \qquad\qquad \forall~ \zeta\in H_0^{1}(\Omega). 
 \label{eq:auxpb2}
\end{align}
We have $\xi =0$ since $\bs s$ is divergence free, and letting $\bs \chi = \bs s$ yields
\begin{align*}
\|\bs s\|_0^2 & = (\curl \bs r,\curl \bs s) 
= \big(\curl \bs r, \curl (\bs s + \bs u_h - \bs v_h)\big) + 
\big(\curl \bs r, \curl (-\bs u_h + \bs v_h)\big)
\\
& = -\big(\curl \bs r, \curl (\nabla q_h)\big) + 
\big(\curl \bs r, \curl (\bs u -\bs u_h )\big) - \big(\curl \bs r, \curl (\bs u -\bs v_h )\big)
.
\end{align*}
By an embedding result (see \cite[Chapter 1 Section 3.4]{Girault;Raviart:1986Finite}), the N\'{e}d\'{e}lec nodal interpolation $\bs I_h^{c_1} \bs r$ is well-defined, inserting which into the first above, letting $\bs \chi = \bs u - \bs v_h$ in \eqref{eq:auxpb1}--\eqref{eq:auxpb2}, and by \eqref{eq:orthogonality-uuh}, we have 
\begin{align*}
\|\bs s\|_0^2
& = \big(\curl \bs e_h , \curl (\bs r - \bs I_h^{c_1} \bs r) \big)
+ (\curl \bs e_h , \curl \bs I_h^{c_1} \bs r) - (\bs s, \bs u - \bs v_h)
\\
& =  \big(\curl \bs e_h , \curl (\bs r - \bs I_h^{c_1} \bs r) \big)
+ (\bs \phi - \bs\phi_h, \curl \bs I_h^{c_1} \bs r) - (\bs s, \bs u - \bs v_h)
\\
&
\leq \|\curl (\bs r - \bs I_h^{c_1} \bs r)\|_0 \|\curl \bs e_h\|_0  + 
\|\bs \phi - \bs\phi_h\|_0 \|\curl \bs I_h^{c_1} \bs r\|_0 + \|\bs s\|_0 \|\bs u - \bs v_h\|_0.
\end{align*}
By standard approximation and stability estimates for the nodal interpolation, as well as a regularity estimate for problem \eqref{eq:auxpb1}--\eqref{eq:auxpb2}, we have
\[
\|\curl (\bs r - \bs I_h^{c_1} \bs r)\|_0 \lesssim h|\curl \bs r|_1\lesssim h\|\bs s\|_0
\quad \text{and } \|\curl \bs I_h^{c_1} \bs r\|_0 \lesssim \|\curl \bs r\|_1 \lesssim \|\bs s\|_0.
\]
As a result, we have
\[
\|\bs s\|_0 \lesssim h\|\curl \bs e_h\|_0 + \|\bs \phi - \bs\phi_h\|_0 + \|\bs u - \bs v_h\|_0.
\]
Lastly, the desired estimate follows from combining the estimates for $\|\bs s_h\|_0$ and $\|\bs s\|_0$ into \eqref{eq:estimateL2-uuh}.
\end{proof} 

%
%
% or use the quadratic first family Nedelec element to improve the $\curl$-norm to second order. In other words, we can treat the last Maxwell equation as a post-processing to get a better approximation $\bs u_h$ from $\bs \phi_h$.

%\LC{The $L^2$ error estimate of Maxwell equation is subtle. Check \cite{Zhong;Shu;Wittum;Xu:2009Optimal} where Liuqiang corrected a mistake in Hiptmair's $L^2$ error estimate. Not sure if it can be adapt to here}
%
%\LC{Also in the adaptivity, we can skip the last system and focus the first two. Once we get optimal order convergence for $\bs \phi$ and want a better $L^2$ approximation of $\bs u$, we solve the last Maxwell equation. }

\subsection{Numerical verification}
\label{sec:numerics-uniform}
In this section, we verify the a priori convergence results shown in the previous subsection. The first example has a smooth solution $\bs{u}(x,y,z) = \left\langle 0,0, (\sin x \,\sin y)^2 \sin z\right\rangle$ on $\Omega = (0,\pi)^3$. 
Because the true solution is not divergence free, problem \eqref{eq:quarticcurldecouple6} needs to be modified to $(\bs{u},\nabla \zeta) = (g, \zeta)$ with $g = \div \bs{u}$ being computed from the true solution, and the discretization changes accordingly. The domain $\Omega$ is partitioned into a uniform tetrahedral mesh, and the convergence plot is in Figure \ref{fig:ex1-uniform}. It can be seen that when $\bs{u}$ and $\curl \bs{u}$ are smooth, the rates of convergence of $|\boldsymbol{\phi}_I - \boldsymbol{\phi}_h|_{1,h}$ and $\|\bs{\phi} - \bs{\phi}_h\|_0$ are optimal, being $O(h)$ and $O(h^2)$, respectively. For the solution $u_h$ obtained from the last Maxwell equation, $\|\curl (\bs u-\bs u_h)\|_0$ is still $O(h)$ and the $L^2$ error $\| \bs u-\bs u_h \|_0$ is improved to $O(h^2)$.
\begin{figure}[htbp]
\begin{center}
\begin{subfigure}[b]{0.45\linewidth}
    \centering
    \includegraphics[width=0.9\textwidth]{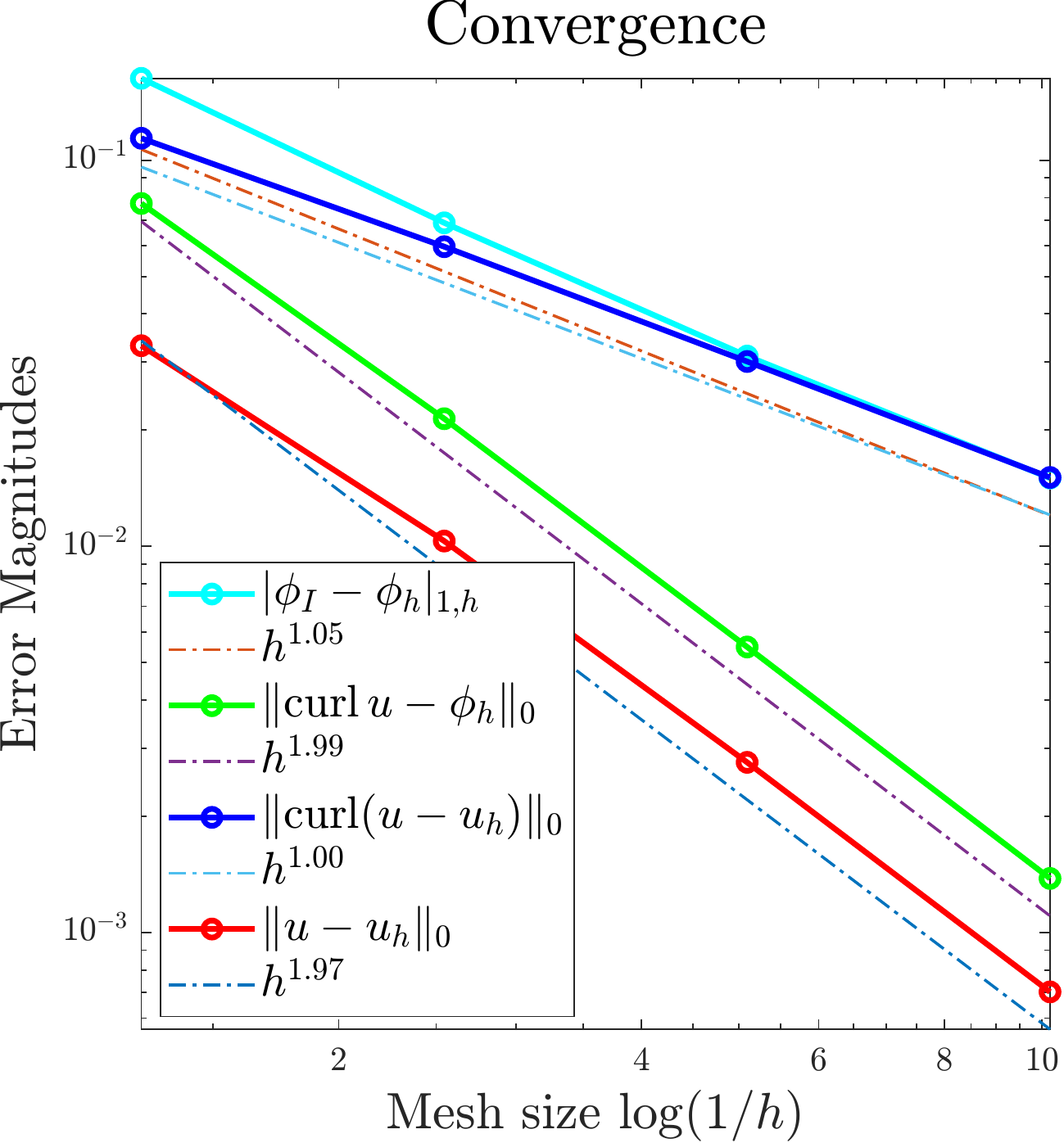}
    \caption{\label{fig:ex1-uniform}}
\end{subfigure}%
\hspace{0.1in}
\begin{subfigure}[b]{0.45\linewidth}
      \centering
      \includegraphics[width=0.9\textwidth]{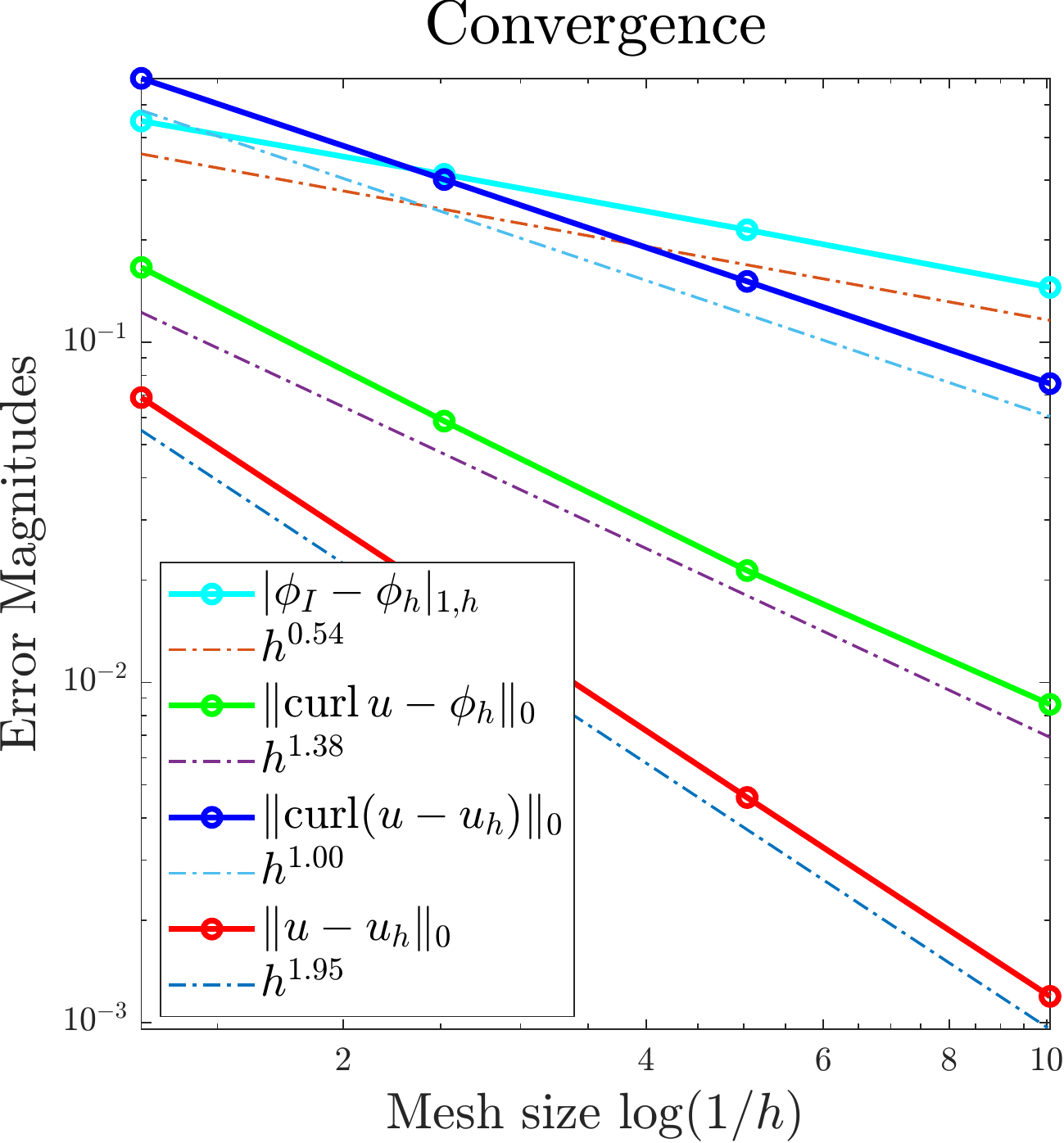}
      \caption{\label{fig:ex2-uniform}}
\end{subfigure}
\caption{On a uniformly refined mesh: (\subref{fig:ex1-uniform}) The convergence of approximating $\bs{u}(x,y,z) = \left\langle 0,0, (\sin x \,\sin y)^2 \sin z\right\rangle$. (\subref{fig:ex2-uniform}) The convergence of approximating 
$\bs{u}(x,y,z) = \curl \left\langle 0,0, r^{8/3}\sin(2\theta/3)\right\rangle$.}
\end{center}
\end{figure}

To demonstrate how the regularity of $\curl \bs{u}$ which is present in \eqref{eq:phihL2error}--\eqref{eq:phihL2error-rate} shall affect we choose a singular solution on an L-shaped domain (Figure \ref{fig:ex2-u}). The true solution is $\bs{u} = \curl \langle 0, 0, \mu\rangle$ for a potential function $\mu = r^{8/3}\sin(2\theta/3)$ in the cylindrical coordinate on $\Omega = (1,1)^2 \times (0,1/2) \backslash ([0,1] \times[-1,0] \times[0,1/2])$. It can be verified that $\mu$ is bi-harmonic so that $\bs{f}=\bs{0}$, and $\curl \bs{u}\in \bs{H}^{5/3-\epsilon}(\Omega; \mathbb R^3)$. The convergence of the approximation $\boldsymbol{\phi} = \curl \bs{u}$ in $|\cdot|_{1,h}$ and $\|\cdot\|_0$ are both sub-optimal (Figure \ref{fig:ex2-uniform}) because $\bs\phi = \curl \bs u\not\in \bs{H}^2(\Omega; \mathbb R^3)$ which is required to achieve the optimal rate of convergence (see Theorem \ref{theorem:errorpriorphih} and Remark \ref{rmk:nonhomogeneous-bc}). While the approximation for $\bs{u}$ is optimal as \eqref{eq:errorprioruh}'s dependence only on the $L^2$-error $\|\bs\phi - \bs\phi_h\|_0$ and the approximation property of the linear N\'{e}d\'{e}lec space for $\bs u\in \bs H^{8/3-\epsilon}(\Omega; \mathbb R^3)$.
\begin{figure}[htbp]
  \centering
\begin{subfigure}[b]{0.4\linewidth}
    \centering
    \includegraphics[width=0.7\textwidth]{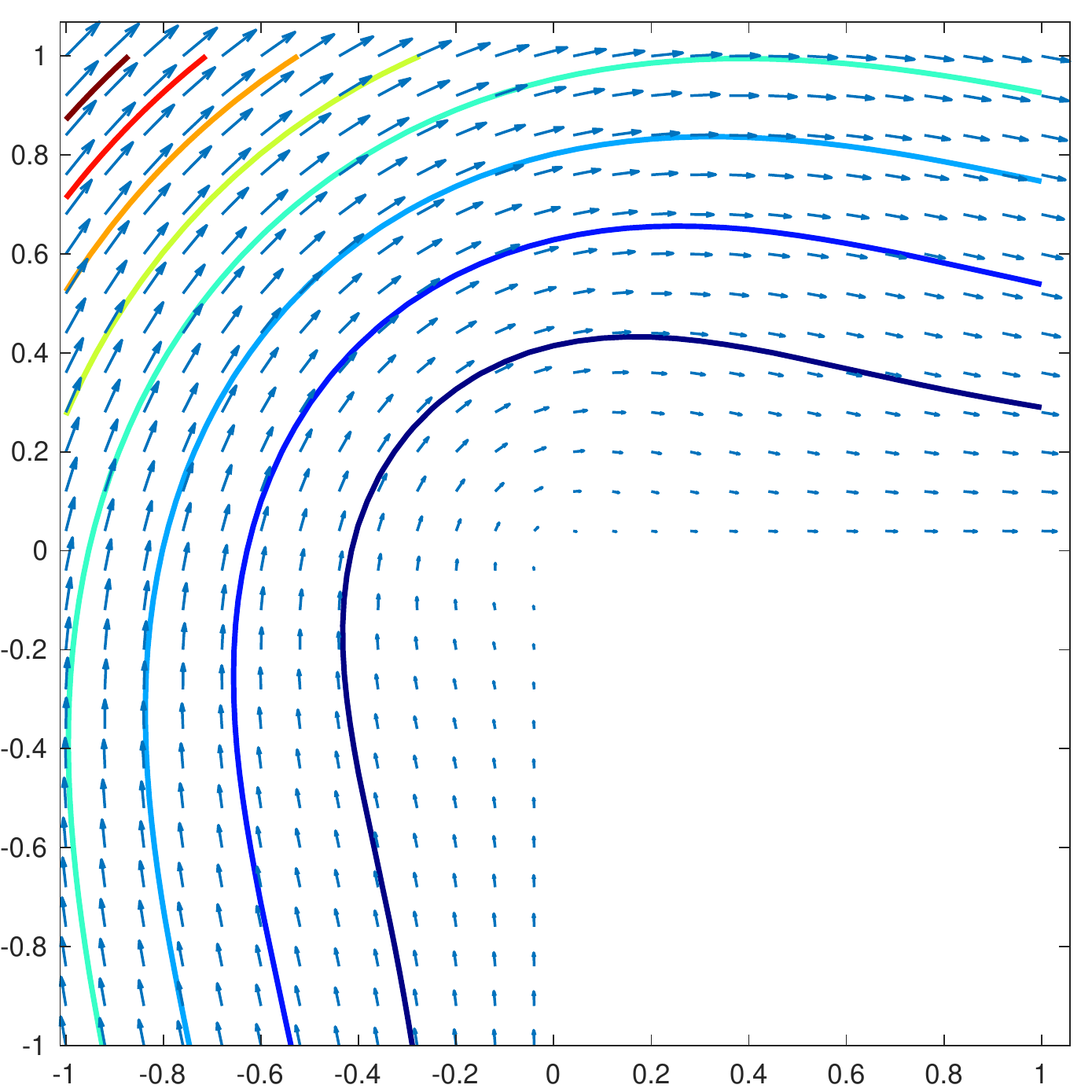}
    \caption{\label{fig:ex2-uexact}}
\end{subfigure}%
\hspace{0.3in}
\begin{subfigure}[b]{0.45\linewidth}
      \centering
      \includegraphics[width=0.7\textwidth]{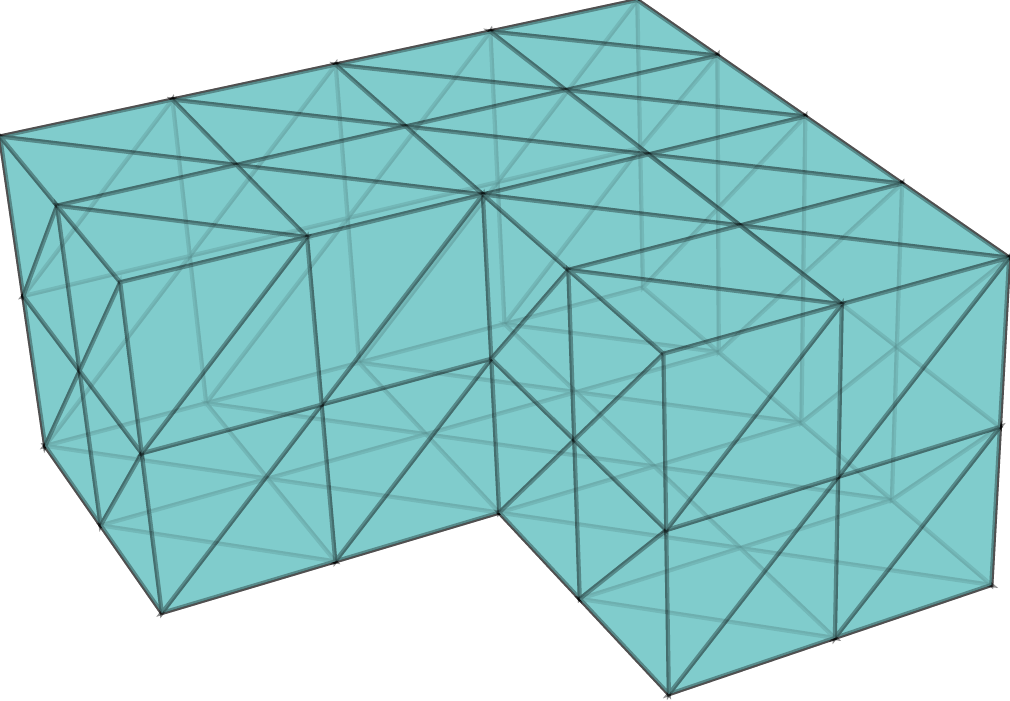}
      \caption{\label{fig:ex2-mesh}}
\end{subfigure}
%\vspace{-20pt}
\caption{The true solution vector field shown in (\subref{fig:ex2-uexact}) of 
the L-shaped domain example viewed from above on $z=1/4$ plane together with the level 
set of its $z$-component. A coarse mesh ($h=1/2$) can be found in (\subref{fig:ex2-mesh}). }
\label{fig:ex2-u}
\end{figure}

\section{A posteriori error analysis}
\label{sec:estimator}
In this section we will propose a reliable and efficient error estimator for the decoupled mixed finite element method \eqref{eq:mfem1}-\eqref{eq:mfem4}. We aim to get an accurate approximation of $\bs u$ in the energy norm which can be controlled by $\| \curl \boldsymbol{u} - \boldsymbol{\phi}_h \|_0 = \|\boldsymbol{\phi} - \boldsymbol{\phi}_h\|_0 \lesssim | \bs \phi - \bs \phi_h|_{1,h}$. Therefore we do not include problem \eqref{eq:mfem5}-\eqref{eq:mfem6} into the adaptive procedure. 

To this end, we first recall a quasi-interpolation \cite{Schoberl2008,ChenWu2017,DemlowHirani2014} and a decomposition of tensor-valued functions \cite{DariDuranPadra1995}.
\begin{lemma}[Theorem 1 in \cite{Schoberl2008}]\label{lem:quasiinterpolationcurl}
There exists an operator $\bs\Pi_h:\bs H_0(\curl, \Omega)\to \bs V_h^c$ such that for any $\bs v\in \bs H_0(\curl, \Omega)$ there exist $\tau\in H_0^1(\Omega)$ and $\bs\chi\in \boldsymbol H_0^1(\Omega; \mathbb{R}^3)$ satisfying
\[
\bs v-\bs\Pi_h\bs v=\nabla\tau+\bs\chi,
\]
\begin{equation}\label{eq:quasiinterpolationcurl}
\sum_{K\in\mathcal T_h}(h_K^{-2}\|\bs\chi\|_{0,K}^2+h_K^{-1}\|\bs\chi\|_{0,\partial K}^2)\lesssim \|\curl\bs v\|_{0}^2.
\end{equation}
%\begin{equation}\label{eq:quasiinterpolationcurl}
%\sum_{K\in\mathcal T_h}(h_K^{-4}\|\bs\chi\|_{0,K}^2+h_K^{-3}\|\bs\chi\|_{0,\partial K}^2)\lesssim \sum_{K\in\mathcal T_h}h_K^{-2}\|\curl\bs v\|_{0,K}^2.
%\end{equation}
\end{lemma}
\begin{lemma}[Lemma~3.2 in \cite{DariDuranPadra1995}]\label{lem:decompL2tensor}
Let $\bs\tau$ be a tensor-valued function in $\bs L^2(\Omega;\mathbb M)$.
There exist $\bs r\in\boldsymbol H_0^1(\Omega; \mathbb{R}^3)$, $q\in L_0^2(\Omega)$, $\bs s\in\boldsymbol H^1(\Omega; \mathbb{M})$ and $\bs v\in\boldsymbol H_0^2(\Omega; \mathbb{R}^3)$ such that
\[
\bs \tau=\bs\nabla\bs r-q\bs I+\bs\curl\bs s,\quad \bs r=\curl\bs v,\quad q=\tr(\bs\curl\bs s),
\]
\begin{equation*}%\label{eq:decompL2tensor}
\|\bs r\|_1+\|\bs s\|_1 +\|q\|_0+\|\bs v\|_2\lesssim \|\bs \tau\|_0.
\end{equation*}
\end{lemma}

For any subset $\mathcal M_h\subseteq\mathcal T_h$, define error estimators
\begin{equation*}
% \label{eq:eta1}
  \eta_1^2(\bs w_h, \bs f,\mathcal M_h):=\sum_{K\in\mathcal M_h}h_K^2\|\bs f\|_{0,K}^2+\sum_{F\in\mathcal F_h^i(\mathcal M_h)}h_F\|\llbracket(\curl\bs w_h)\times\bs n_F\rrbracket\|_{0,F}^2,
\end{equation*}
\begin{equation*}
% \label{eq:eta2}
  \eta_2^2(\bs\phi_h, \bs w_h,\mathcal M_h):=\sum_{K\in\mathcal M_h}h_K^2\|\curl\bs w_h\|_{0,K}^2+\sum_{F\in\mathcal F_h(\mathcal M_h)}h_F\|\llbracket\bs n_F\times(\bs\nabla_h\bs \phi_h)\rrbracket\|_{0,F}^2.
\end{equation*}
%\[
%\eta_1^2(\bs w_h, \bs f,\mathcal M_h):=\sum_{K\in\mathcal M_h}h_K^4\|\bs f\|_{0,K}^2+\sum_{F\in\mathcal F_h^i(\mathcal M_h)}h_F^3\|\llbracket(\curl\bs w_h)\times\bs n_F\rrbracket\|_{0,F}^2,
%\]
%\[
%\eta_3^2(\bs\phi_h, \bs u_h,\mathcal M_h):=\sum_{K\in\mathcal M_h}h_K^2\|\curl\bs \phi_h\|_{0,K}^2+\sum_{F\in\mathcal F_h^i(\mathcal M_h)}h_F\|\llbracket(\bs \phi_h-\curl\bs u_h)\times\bs n_F\rrbracket\|_{0,F}^2,
%\]
% \[
% \eta^2(\bs\phi_h, \bs u_h,\bs w_h, \bs f,\mathcal M_h):=\eta_1^2(\bs w_h, \bs f,\mathcal M_h)+\eta_2^2(\bs\phi_h, \bs w_h,\mathcal M_h)+\eta_3^2(\bs\phi_h, \bs u_h,\mathcal M_h).
% \]

Let $\bs Q_0^K \bs f$ be the $\bs L^2$-projection of the data onto 
$\prod\limits_{K\in \mathcal{T}_h} \mathbb P_{0}(K;\mathbb R^3)$, then the data oscillation is defined as
\[
\mathrm{osc}^2(\bs f,\mathcal M_h):=\sum_{K\in\mathcal M_h}h_K^2\|\bs f-\bs Q_0^K\bs f\|_{0,K}^2.
\]

Let $\bs I_h^{SZ}$ be the tensorial Scott-Zhang interpolation from $\boldsymbol H^1(\Omega; \mathbb{M})$ to the tensorial linear Lagrange element space \cite{ScottZhang1990}. It holds
\begin{equation}\label{eq:scottzhangestimate}
\sum_{K\in\mathcal T_h}\left(h_K^{-2}\|\bs v-\bs I_h^{SZ}\bs v\|_{0,K}^2+|\bs v-\bs I_h^{SZ}\bs v|_{1,K}^2\right)\lesssim |\bs v|_1^2\quad\forall~\bs v\in \boldsymbol H^1(\Omega; \mathbb{M}).
\end{equation}

%\begin{lemma}\label{lem:errorposterioriwh0}
%Let $(\bs w, 0)\in \bs H_0(\curl, \Omega)\times H_0^{1}(\Omega)$ be the solution of the Maxwell equation \eqref{eq:quarticcurldecouple1}-\eqref{eq:quarticcurldecouple2}, and $(\bs w_h, 0)\in \bs V_h^c\times V_h^1$ the solution of the mixed method \eqref{eq:mfem1}-\eqref{eq:mfem2}.
%For any $\bs v\in \bs H_0(\curl, \Omega)$, it holds
%\begin{equation}\label{eq:errorposterioriwh0}
%(\curl (\bs w-\bs w_h), \curl\bs v)\lesssim \eta_1(\bs w_h, \bs f,\mathcal T_h)\|\curl\bs v\|_{0}.
%\end{equation}
%\end{lemma}
We first present an a posterior analysis of error $\bs w -\bs w_h$ which is well-documented for the saddle point formulation of Maxwell's equation (see e.g., \cite{Beck;Hiptmair;Hoppe;Wohlmuth:2000Residual,Zhong;Shu;Chen;Xu:2010Convergence,Zheng;Chen;Wang:2006adaptive}). We include a proof here for the completeness.
\begin{lemma}\label{lem:errorposterioriwh}
Let $(\bs w, 0)\in \bs H_0(\curl, \Omega)\times H_0^{1}(\Omega)$ be the solution of the Maxwell equation \eqref{eq:quarticcurldecouple1}-\eqref{eq:quarticcurldecouple2}, and $(\bs w_h, 0)\in \bs V_h^c\times V_h^1$ the solution of the mixed method \eqref{eq:mfem1}-\eqref{eq:mfem2}.
We have
\begin{equation}\label{eq:errorposterioriwhreliability}
\|\curl (\bs w-\bs w_h)\|_0\lesssim \eta_1(\bs w_h, \bs f,\mathcal T_h),
\end{equation}
\begin{equation}\label{eq:errorposterioriwhefficiency}
\eta_1(\bs w_h, \bs f,\mathcal T_h)\lesssim \|\curl (\bs w-\bs w_h)\|_0+\mathrm{osc}(\bs f,\mathcal T_h).
\end{equation}
\end{lemma}
\begin{proof}
Applying Lemma~\ref{lem:quasiinterpolationcurl} to $\bs v=\bs w-\bs w_h$, we get from \eqref{eq:GalerkinOrthcurl} and \eqref{eq:quarticcurldecouple1} that
\begin{align*}
\|\curl (\bs w-\bs w_h)\|_0^2&=(\curl (\bs w-\bs w_h), \curl(\bs v-\bs\Pi_h\bs v))=(\curl (\bs w-\bs w_h), \curl\bs\chi) \\
&=(\bs f, \bs\chi)-(\curl\bs w_h, \curl\bs\chi) \\
&=(\bs f, \bs\chi)-\sum_{K\in\mathcal T_h}((\curl\bs w_h)\times\bs n, \bs\chi)_{\partial K} \\
&=(\bs f, \bs\chi)-\sum_{F\in\mathcal F_h^i}(\llbracket(\curl\bs w_h)\times\bs n_F\rrbracket, \bs\chi)_{F}.
\end{align*}
Hence we have derived \eqref{eq:errorposterioriwhreliability} by \eqref{eq:quasiinterpolationcurl}.

The efficiency \eqref{eq:errorposterioriwhefficiency} follows from the standard bubble function techniques (see e.g., \cite{Beck;Hiptmair;Hoppe;Wohlmuth:2000Residual}).
\end{proof}

%\begin{proof}
%Let $\bs v=\bs w-\bs w_h$. Applying Lemma~\ref{lem:quasiinterpolationcurl} to $\bs v$, we get from \eqref{eq:GalerkinOrthcurl} and \eqref{eq:quarticcurldecouple1} that
%\begin{align*}
%\|\curl (\bs w-\bs w_h)\|_0^2&=(\curl (\bs w-\bs w_h), \curl\bs v)=(\curl (\bs w-\bs w_h), \curl(\bs v-\bs\Pi_h\bs v)) \\
%&=(\curl (\bs w-\bs w_h), \curl\bs\chi)=(\bs f, \bs\chi)-(\curl\bs w_h, \curl\bs\chi) \\
%&=(\bs f, \bs\chi)+\sum_{K\in\mathcal T_h}((\curl\bs w_h)\times\bs n, \bs\chi)_{\partial K} \\
%&=(\bs f, \bs\chi)+\sum_{F\in\mathcal F_h^i}(\llbracket(\curl\bs w_h)\times\bs n_F\rrbracket, \bs\chi)_{F}.
%\end{align*}
%Then we derive the reliability \eqref{eq:errorposterioriwhreliability} by using \eqref{eq:quasiinterpolationcurl}.
%The efficiency \eqref{eq:errorposterioriwhefficiency} follows from the standard bubble function techniques.
%\end{proof}

%As a direct result of Lemma~\ref{lem:errorposterioriwh0}, we have the following a posteriori error estimate for $\bs w_h$.
%\begin{lemma}
%Let $(\bs w, 0)\in \bs H_0(\curl, \Omega)\times H_0^{1}(\Omega)$ be the solution of the Maxwell equation \eqref{eq:quarticcurldecouple1}-\eqref{eq:quarticcurldecouple2}, and $(\bs w_h, 0)\in \bs V_h^c\times V_h^1$ the solution of the mixed method \eqref{eq:mfem1}-\eqref{eq:mfem2}.
%For any $\bs v\in \bs H_0(\curl, \Omega)$, it holds
%\begin{equation*}%\label{eq:errorposterioriwh0}
%\sum_{K\in\mathcal T_h}h_K^{2}\|\curl (\bs w-\bs w_h)\|_{0,K}^2\lesssim \eta_1^2(\bs w_h, \bs f,\mathcal T_h).
%\end{equation*}
%\end{lemma}

\begin{lemma}\label{lem:errorposterioriphihph}
Let $(\bs w, 0, \bs \phi, p)\in \bs H_0(\curl, \Omega)\times H_0^{1}(\Omega)\times\boldsymbol H_0^1(\Omega; \mathbb{R}^3)\times L_0^{2}(\Omega)$ be the solution of the variational formulation \eqref{eq:quarticcurldecouple1}-\eqref{eq:quarticcurldecouple4}, and $(\bs w_h, 0, \bs \phi_h, p_h)\in \bs V_h^c\times V_h^1\times\bs V_h^{\rm CR}\times \mathcal Q_h$ the solution of the mixed method \eqref{eq:mfem1}-\eqref{eq:mfem4}.
We have
\begin{equation}\label{eq:errorposterioriphihph}
|\bs \phi-\bs \phi_h|_{1,h}+ h\|\curl (\bs w-\bs w_h)\|_0 \lesssim h\eta_1(\bs w_h, \bs f,\mathcal T_h)+\eta_2(\bs\phi_h, \bs w_h,\mathcal T_h),
\end{equation}
\begin{align}
h\eta_1(\bs w_h, \bs f,\mathcal T_h)+\eta_2(\bs\phi_h, \bs w_h,\mathcal T_h)\lesssim &~|\bs \phi-\bs \phi_h|_{1,h}+\|p-p_h\|_0 \notag\\
&+h\|\curl (\bs w-\bs w_h)\|_0+h\mathrm{osc}(\bs f,\mathcal T_h). \label{eq:errorposterioriphihphefficientcy}
\end{align}
\end{lemma}
\begin{proof}
Applying Lemma~\ref{lem:decompL2tensor} to $\bs\nabla_h(\bs \phi-\bs \phi_h)$,
there exist $\bs r\in\boldsymbol H_0^1(\Omega; \mathbb{R}^3)$, $q\in L_0^2(\Omega)$, $\bs s\in\boldsymbol H^1(\Omega; \mathbb{M})$ and $\bs v\in\boldsymbol H_0^2(\Omega; \mathbb{R}^3)$ such that
\[
\bs\nabla_h(\bs \phi-\bs \phi_h)=\bs\nabla\bs r-q\bs I+\bs\curl\bs s,\quad \bs r=\curl\bs v,
\]
\begin{equation}\label{eq:decompL2tensor}
\|\bs r\|_1+\|\bs s\|_1 +\|q\|_0+\|\bs v\|_2\lesssim \|\bs\nabla_h(\bs \phi-\bs \phi_h)\|_0.
\end{equation}
Note that $\div_h(\bs \phi-\bs \phi_h)=0$ from \eqref{eq:quarticcurldecouple4} and \eqref{eq:mfem4}. Since $\div\bs r=0$, we get from \eqref{eq:quarticcurldecouple3} that
\begin{align*}
|\bs \phi-\bs \phi_h|_{1,h}^2&=(\bs\nabla_h(\bs \phi-\bs \phi_h), \bs\nabla\bs r-q\bs I+\bs\curl\bs s)=(\bs\nabla_h(\bs \phi-\bs \phi_h), \bs\nabla\bs r+\bs\curl\bs s) \\
&=(\bs\nabla_h(\bs \phi-\bs \phi_h), \bs\nabla\bs r)+(\div\bs r, p-p_h) + (\bs\nabla_h(\bs \phi-\bs \phi_h), \bs\curl\bs s) \\
&=(\curl  \bs w, \bs r)-(\bs\nabla_h\bs \phi_h, \bs\nabla\bs r)-(\div\bs r, p_h) - (\bs\nabla_h\bs \phi_h, \bs\curl\bs s). %\label{eq:20200201-1}
\end{align*}
It follows from \eqref{eq:errorestimateIhs1} and \eqref{eq:mfem3} that
\[
(\bs\nabla_h\bs \phi_h, \bs\nabla\bs r)+(\div\bs r, p_h)=(\bs\nabla_h\bs \phi_h, \bs\nabla(\bs I_h^{s}\bs r))+(\div(\bs I_h^{s}\bs r), p_h)=(\curl\bs w_h, \bs I_h^{s}\bs r).
\]
Noticing that $(\bs\nabla_h\bs \phi_h, \bs\curl(\bs I_h^{SZ}\bs s))=0$, we obtain from the last two identities that
\begin{align}
|\bs \phi-\bs \phi_h|_{1,h}^2&=(\curl  \bs w, \bs r)-(\curl\bs w_h, \bs I_h^{s}\bs r) - (\bs\nabla_h\bs \phi_h, \bs\curl(\bs s-\bs I_h^{SZ}\bs s))\notag\\
&=(\curl(\bs w-\bs w_h), \curl\bs v)+(\curl\bs w_h, \bs r-\bs I_h^{s}\bs r) \notag\\
&\quad - (\bs\nabla_h\bs \phi_h, \bs\curl(\bs s-\bs I_h^{SZ}\bs s)). \label{eq:20200201-2}
\end{align}
Next we estimate the right hand side of \eqref{eq:20200201-2} term by term.
%Employing \eqref{eq:GalerkinOrthcurl} and \eqref{eq:errorposterioriwh0}, it follows
Employing \eqref{eq:GalerkinOrthcurl} and \eqref{eq:errorposterioriwhreliability}, it follows
%\begin{align*}
%(\curl(\bs w-\bs w_h), \curl\bs v)&=\inf_{\bs v_h\in \bs V_h^c}(\curl(\bs w-\bs w_h), \curl(\bs v-\bs v_h)) \\
%&\lesssim \eta_1(\bs w_h, \bs f,\mathcal T_h)\inf_{\bs v_h\in \bs V_h^c}\left(\sum_{K\in\mathcal T_h}h_K^{-2}\|\curl(\bs v-\bs v_h)\|_{0,K}^2\right)^{1/2} \\
%&\lesssim \eta_1(\bs w_h, \bs f,\mathcal T_h)|\curl\bs v|_{1}=\eta_1(\bs w_h, \bs f,\mathcal T_h)|\bs r|_{1}.
%\end{align*}
\begin{align*}
(\curl(\bs w-\bs w_h), \curl\bs v)&=\inf_{\bs v_h\in \bs V_h^c}(\curl(\bs w-\bs w_h), \curl(\bs v-\bs v_h)) \\
&\leq \|\curl(\bs w-\bs w_h)\|_0\inf_{\bs v_h\in \bs V_h^c}\|\curl(\bs v-\bs v_h)\|_0 \\
&\lesssim h\eta_1(\bs w_h, \bs f,\mathcal T_h)|\curl\bs v|_{1}= h\eta_1(\bs w_h, \bs f,\mathcal T_h)|\bs r|_{1}.
\end{align*}
According to \eqref{eq:errorestimateIhs}, we have
\[
(\curl\bs w_h, \bs r-\bs I_h^{s}\bs r)\lesssim \eta_2(\bs\phi_h, \bs w_h,\mathcal T_h)|\bs r|_1.
\]
And we get from \eqref{eq:scottzhangestimate} that
\begin{align*}
-(\bs\nabla_h\bs \phi_h, \bs\curl(\bs s-\bs I_h^{SZ}\bs s))&=\sum_{K\in\mathcal T_h}(\bs n\times(\bs\nabla_h\bs \phi_h), \bs s-\bs I_h^{SZ}\bs s)_{\partial K}  \\
&=\sum_{F\in\mathcal F_h}(\llbracket\bs n_F\times(\bs\nabla_h\bs \phi_h)\rrbracket, \bs s-\bs I_h^{SZ}\bs s)_{F} \\
&\lesssim \eta_2(\bs\phi_h, \bs w_h,\mathcal T_h)|\bs s|_1.
\end{align*}
Combining the last three inequalities and \eqref{eq:20200201-2}, we get from \eqref{eq:decompL2tensor} that
\begin{equation*}%\label{eq:errorposterioriphih}
|\bs \phi-\bs \phi_h|_{1,h} \lesssim h \eta_1(\bs w_h, \bs f,\mathcal T_h)+\eta_2(\bs\phi_h, \bs w_h,\mathcal T_h),
\end{equation*}
which together with \eqref{eq:errorposterioriwhreliability} indicates \eqref{eq:errorposterioriphihph}.

On the other side, by applying the bubble function techniques, we get
\begin{equation*}%\label{eq:curlwhefficiency}
\sum_{K\in\mathcal T_h}h_K^2\|\curl\bs w_h\|_{0,K}^2\lesssim h^2\|\curl (\bs w-\bs w_h)\|_0^2+|\bs \phi-\bs \phi_h|_{1,h}^2+\|p-p_h\|_0^2,
\end{equation*}
\[
\sum_{F\in\mathcal F_h}h_F\|\llbracket\bs n_F\times(\bs\nabla_h\bs \phi_h)\rrbracket\|_{0,F}^2\lesssim |\bs \phi-\bs \phi_h|_{1,h}^2.
\]
Combining the last two inequalities shows
\[
\eta_2(\bs\phi_h, \bs w_h,\mathcal T_h)\lesssim h\|\curl (\bs w-\bs w_h)\|_0+|\bs \phi-\bs \phi_h|_{1,h}+\|p-p_h\|_0.
\]
Therefore we conclude \eqref{eq:errorposterioriphihphefficientcy} from \eqref{eq:errorposterioriwhefficiency}.
\end{proof}

% \LC{We have given the L2 error estimate before. Simply change the index.}
By the a priori $L^2$-estimate of the Stokes problem, when we assume the $H^{1+s}$-regularity ($s\in (1/2,1]$) for the possible non-smooth solution, the following estimate, combining with  Lemmas \ref{lem:errorposterioriwh} and \ref{lem:errorposterioriphihph}, can be used for a global reliability bound for $\|\bs \phi-\bs \phi_h\|_0$.
\begin{lemma}\label{lem:errorposterioriphihphL2}
Let $(\bs w, 0, \bs \phi, p)\in \bs H_0(\curl, \Omega)\times H_0^{1}(\Omega)\times\boldsymbol H_0^1(\Omega; \mathbb{R}^3)\times L_0^{2}(\Omega)$ be the solution of the variational formulation \eqref{eq:quarticcurldecouple1}-\eqref{eq:quarticcurldecouple4}, and $(\bs w_h, 0, \bs \phi_h, p_h)\in \bs V_h^c\times V_h^1\times\bs V_h^{\rm CR}\times \mathcal Q_h$ the solution of the mixed method \eqref{eq:mfem1}-\eqref{eq:mfem4}.
We have
\begin{equation*}%\label{eq:errorposterioriphihphL2}
\|\bs \phi-\bs \phi_h\|_0\lesssim h^s|\bs \phi-\bs \phi_h|_{1,h}+h\|\curl(\bs w-\bs w_h)\|_0
+h^{1+s}\|\curl\bs w_h\|_0.
\end{equation*}
\end{lemma}

\section{Quasi-orthogonality}
\label{sec:quasi-orthogonality}

In this section we will develop the quasi-orthogonality of the decoupled mixed finite element method.

%\subsection{Nonconforming discrete Stokes complex in three dimensions}
\subsection{Discrete complexes in three dimensions}
First recall a nonconforming discretization of the following Stokes complex in three dimensions \cite{Huang2020} %\cite[(4.16)]{JohnLinkeMerdonNeilanEtAl2017}
%\cite{EvansHughes2013}
\begin{equation*}%\label{eq:Stokescomplex3d}
0\autorightarrow{}{} H_0^1(\Omega)\autorightarrow{$\nabla$}{} \bs H_0(\grad\curl,\Omega) \autorightarrow{$\curl$}{} \boldsymbol H_0^1(\Omega; \mathbb{R}^3) \autorightarrow{$\div$}{} L_0^{2}(\Omega)\autorightarrow{}{}0,
\end{equation*}
where $\bs H_0(\grad\curl,\Omega):=\{\bs v\in\bs H_0(\curl, \Omega): \curl\bs v\in \boldsymbol H_0^1(\Omega; \mathbb{R}^3)\}$. Note that $\bs H_0(\grad\curl,\Omega)=\bs H_0(\curlcurl, \Omega)$ in \eqref{eq:space-h2curl} (cf. \cite{Zhang2018a}).
%We refer to \cite{BeiraodaVeigaDassiVacca2019} for a conforming virtual element Stokes complex based on the complex \eqref{eq:Stokescomplex3d}.

The space of the shape functions of the $\bs H(\grad\curl)$ nonconforming element proposed in \cite{Huang2020} is $\mathbb P_{0}(K;\mathbb R^3) \oplus \bs x\wedge\mathbb P_{1}(K;\mathbb R^3)$, and the local degrees of freedom are given by
\begin{align}
\int_e\bs v\cdot\bs t_e\dd s & \quad \textrm{ on each }  e\in\mathcal E(K), \label{dof1}\\
\int_F(\curl\bs v)\cdot\bs t_{F,i}\dd s & \quad \textrm{ on each }  F\in\mathcal F(K) \textrm{ with } i=1,2. \label{dof2} 
\end{align}
The global $\bs H(\grad\curl)$ nonconforming element space is then defined as
\begin{align*}
\bs W_h:=\{\bs v_h\in \boldsymbol L^2(\Omega; \mathbb{R}^3):&\, \bs v_h|_K\in\mathbb P_{0}(K;\mathbb R^3) \oplus \bs x\wedge\mathbb P_{1}(K;\mathbb R^3) \textrm{ for each } K\in\mathcal T_h, \\
&\textrm{ all the degree of freedom \eqref{dof1}-\eqref{dof2} are single-valued}, \\
&\textrm{ and all the degree of freedom \eqref{dof1}-\eqref{dof2} on $\partial \Omega$ vanish}\}.
\end{align*}
Now the nonconforming discrete Stokes complex in \cite{Huang2020} is presented as
\begin{equation}\label{eq:Stokescomplex3dncfem}
0\autorightarrow{}{} V_h^1\autorightarrow{$\nabla$}{} \bs W_h \autorightarrow{$\curl_h$}{} \bs V_h^{\rm CR} \autorightarrow{$\div_h$}{}  \mathcal Q_h\autorightarrow{}{}0.
\end{equation}

%\begin{equation}\label{eq:cdhermiteHd03dintro}
%\resizebox{.91\hsize}{!}{$
%\begin{array}{c}
%\xymatrix{
%0 \ar[r]^{} & \ar[d]^{\Pi_h^{\grad}}  H_0^{1}(\Omega) \ar[r]^-{\nabla} & H_0(\curl,\Omega) \ar[d]^{\Pi_h^{\curl}} \ar[r]^-{\curl}
%                & H_0(\div,\Omega) \ar[d]^{\Pi_h^{\div}}   \ar[r]^-{\div} & \ar[d]^{\Pi_h^{0}}L_0^2(\Omega) \ar[r]^{} & 0 \\
%0 \ar[r]^{} &  \mathring{H}_{k+3}^{\grad}(\mathcal T_h) \ar[r]^-{\nabla} & \mathring{H}_{k+2}^{\curl}(\mathcal T_h) \ar[r]^-{\curl}
%                & \mathring{H}_{k+1}^{\div}(\mathcal T_h)   \ar[r]^-{\div} &  \mathring{H}_{k}^{0}(\mathcal T_h) \ar[r]^{}& 0    }
%\end{array}
%$}
%\end{equation}

We also need the help of a discrete de Rham complex.
Recall the lowest-order Raviart-Thomas element space \cite{RaviartThomas1977,Nedelec1980}
\begin{equation*}
%\label{eq:space-rt}
  \bs V_h^d:=\{\bs v_h\in \bs H_0(\div, \Omega): \bs v_h|_K\in\mathbb P_{0}(K;\mathbb R^3)+\bs x\mathbb P_{0}(K) \textrm{ for each } K\in\mathcal T_h\},
\end{equation*}
and the discrete de Rham complex \cite{Arnold2018,ArnoldFalkWinther2006}
\[
0\autorightarrow{}{} V_h^1\autorightarrow{$\nabla$}{} \bs V_h^c \autorightarrow{$\curl$}{} \bs V_h^d \autorightarrow{$\div$}{}  \mathcal Q_h\autorightarrow{}{}0.
\]
Let $\bs I_h^{c}$ be the nodal interpolation operator from $\textrm{Dom}(\bs I_h^{c})$ to $\bs V_h^c$, and $\bs I_h^{d}$ the nodal interpolation operator from $\textrm{Dom}(\bs I_h^{d})$ to $\bs V_h^d$, where $\textrm{Dom}(\bs I_h^{c})$ and $\textrm{Dom}(\bs I_h^{d})$ are the domains of the operators $\bs I_h^{c}$ and $\bs I_h^{d}$ respectively.
It holds for any $\bs v_h\in \bs W_h$ that
\begin{equation*}%\label{eq:errorestimateIcW}
\|\curl_h(\bs v_h-\bs I_h^{c}\bs v_h)\|_{0,K}=\|\curl_h\bs v_h-\bs I_h^{d}(\curl_h\bs v_h)\|_{0,K}\lesssim h_K|\curl_h\bs v_h|_{1,K}.
\end{equation*}

Henceforth, consider two nested conforming triangulations $\mathcal T_h$ and $\mathcal T_H$, where $\mathcal T_h$ is a refinement of $\mathcal T_H$.
We have the commutative diagram property \cite{Arnold2018,ArnoldFalkWinther2006}
\begin{equation}\label{eq:cd2grid}
\curl(\bs I_H^{c}\bs v_h)= \bs I_H^{d}(\curl_h\bs v_h)\quad\forall~\bs v_h\in \bs V_h^c+\bs W_h.
\end{equation}
To derive the quasi-orthogonality,  we need the following interpolation error estimation for $\bs I_H^{s}$ \cite{HuXu2013}
\begin{equation}\label{eq:errorestimate2gridIs}
\|\bs\psi_h-\bs I_H^{s}\bs\psi_h\|_{0,K}\lesssim h_K\|\nabla_h\bs\psi_h\|_{0,K}\quad\forall~\bs\psi_h\in \bs V_h^{\rm CR}, K\in\mathcal T_H\backslash\mathcal T_h.
\end{equation}
%Following the similar idea in \cite{HuXu2013}, we get from \eqref{eq:cd2grid} that
%\begin{equation}\label{eq:errorestimate2gridIc}
%\|\curl(\bs v_h-\bs I_H^{c}\bs v_h)\|_{0,K}\lesssim h_K|\curl\bs v_h|_{1,K}\quad\forall~\bs v_h\in \bs V_h^c, K\in\mathcal T_H\backslash\mathcal T_h.
%\end{equation}

According to \eqref{eq:mfem1}, since the triangulations $\mathcal T_h$ and $\mathcal T_H$ are nested, we get the following Galerkin orthogonality
\begin{equation}\label{eq:GalerkinOrthcurl2grid}
(\curl (\bs w_h-\bs w_H), \curl \bs v_H)=0\quad \forall~\bs v_H\in \bs V_H^c.
\end{equation}

\begin{lemma}\label{lem:IHdVsh}
It holds
\begin{equation}\label{eq:errorestimate2gridId}
\sum_{K\in\mathcal T_H}h_K^{-2}\|\bs\psi_h-\bs I_H^{d}\bs\psi_h\|_{0,K}^2\lesssim |\bs\psi_h|_{1,h}^2\quad\forall~\bs\psi_h\in \bs V_h^{\rm CR}.
\end{equation}
\end{lemma}
\begin{proof}
By the averaging technique \cite{Brenner1996,HuangHuang2011}, there exists $\widetilde{\bs\psi}_h\in\bs V_h^1:=V_h^1\otimes\mathbb R^3$ such that
\begin{equation}\label{eq:20200307-1}
|\widetilde{\bs\psi}_h|_{1}^2+\sum_{K\in\mathcal T_h}\left(h_K^{-2}\|\bs\psi_h-\widetilde{\bs\psi}_h\|_{0,K}^2+h_K^{-1}\|\bs\psi_h-\widetilde{\bs\psi}_h\|_{0,\partial K}^2\right)\lesssim |\bs\psi_h|_{1,h}^2.
\end{equation}
Applying the scaling argument, we get
\begin{align*}
\sum_{K\in\mathcal T_H}h_K^{-2}\|\bs I_H^{d}(\bs\psi_h-\widetilde{\bs\psi}_h)\|_{0,K}^2&\lesssim \sum_{K\in\mathcal T_H}h_K^{-1}\|(\bs\psi_h-\widetilde{\bs\psi}_h)\cdot\bs n\|_{0,\partial K}^2 \\
&\leq \sum_{K\in\mathcal T_H}h_K^{-1}\Bigg(\sum_{K'\in\mathcal T_h(K)}\|\bs\psi_h-\widetilde{\bs\psi}_h\|_{0,\partial K'}^2\Bigg) \\
&\leq \sum_{K\in\mathcal T_h}h_K^{-1}\|\bs\psi_h-\widetilde{\bs\psi}_h\|_{0,\partial K}^2.
\end{align*}
Noting that $\widetilde{\bs\psi}_h\in \boldsymbol{H}_0^1(\Omega;\mathbb R^3)$, it follows
\[
\sum_{K\in\mathcal T_H}h_K^{-2}\|\widetilde{\bs\psi}_h-\bs I_H^{d}\widetilde{\bs\psi}_h\|_{0,K}^2\lesssim |\widetilde{\bs\psi}_h|_{1,h}^2.
\]
Combining the last two inequalities and \eqref{eq:20200307-1} yields
\begin{align*}
\sum_{K\in\mathcal T_H}h_K^{-2}\left(\|\bs I_H^{d}(\bs\psi_h-\widetilde{\bs\psi}_h)\|_{0,K}^2+\|\bs\psi_h-\widetilde{\bs\psi}_h\|_{0,K}^2+\|\widetilde{\bs\psi}_h-\bs I_H^{d}\widetilde{\bs\psi}_h\|_{0,K}^2\right)\lesssim |\bs\psi_h|_{1,h}^2.
\end{align*}
On the other hand, we have
\[
\|\bs\psi_h-\bs I_H^{d}\bs\psi_h\|_{0,K}\leq \|\bs I_H^{d}(\bs\psi_h-\widetilde{\bs\psi}_h)\|_{0,K}+\|\bs\psi_h-\widetilde{\bs\psi}_h\|_{0,K}+\|\widetilde{\bs\psi}_h-\bs I_H^{d}\widetilde{\bs\psi}_h\|_{0,K}.
\]
%\begin{align*}
%\|\bs\psi_h-\bs I_H^{d}\bs\psi_h\|_{0,K}&\leq \|\bs\psi_h-\widetilde{\bs\psi}_h-\bs I_H^{d}(\bs\psi_h-\widetilde{\bs\psi}_h)\|_{0,K}+\|\widetilde{\bs\psi}_h-\bs I_H^{d}\widetilde{\bs\psi}_h\|_{0,K} \\
%&\leq \|\bs I_H^{d}(\bs\psi_h-\widetilde{\bs\psi}_h)\|_{0,K}+\|\bs\psi_h-\widetilde{\bs\psi}_h\|_{0,K}+\|\widetilde{\bs\psi}_h-\bs I_H^{d}\widetilde{\bs\psi}_h\|_{0,K}.
%\end{align*}
Therefore we conclude \eqref{eq:errorestimate2gridId} from the last two inequalities.
\end{proof}

\subsection{Quasi-orthogonality}

\begin{lemma}[quasi-orthogonality of $\bs\phi$]
Let $(\bs w_h, 0, \bs \phi_h, p_h)\in \bs V_h^c\times V_h^1\times\bs V_h^{\rm CR}\times \mathcal Q_h$ and $(\bs w_H, 0, \bs \phi_H, p_H)\in \bs V_H^c\times V_h^1\times\bs V_h^{\rm CR}\times \mathcal Q_H$ the solutions of the mixed method \eqref{eq:mfem1}-\eqref{eq:mfem4} on triangulations $\mathcal T_h$ and $\mathcal T_H$ respectively.
We have
\begin{align}
(\bs\nabla_h(\bs\phi-\bs\phi_h), \bs\nabla_h(\bs\phi_h-\bs\phi_H))\lesssim & H|\bs\phi-\bs\phi_h|_{1,h}\|\curl(\bs w_h-\bs w_H)\|_0 \notag\\
 &+ |\bs\phi-\bs\phi_h|_{1,h}\Bigg(\sum_{K\in\mathcal T_H\backslash\mathcal T_h}h_K^2\|\curl\bs w_H\|_{0,K}^2\Bigg)^{1/2}. \label{eq:phiHphihORerror}
\end{align}
\end{lemma}

\begin{proof}
Let $\bs\psi_h=\bs I_h^{s}(\bs\phi-\bs\phi_h)$, then we get from \eqref{eq:errorestimateIhs1} and \eqref{eq:errorestimateIhs} that
\begin{equation}\label{eq:20200307-3}
\div_h\bs\psi_h=0,\quad |\bs\psi_h|_{1,h}\lesssim |\bs\phi-\bs\phi_h|_{1,h}.
\end{equation}
It follows from \eqref{eq:mfem4}, \eqref{eq:errorestimateIhs1} and \eqref{eq:mfem3} that
\begin{align*}
(\bs\nabla_h(\bs\phi-\bs\phi_h), \bs\nabla_h(\bs\phi_h-\bs\phi_H))&=(\bs\nabla_h(\bs\phi-\bs\phi_h), \bs\nabla_h\bs\phi_h+p_h\bs I-(\bs\nabla_H\bs\phi_H+p_H\bs I)) \\
&=(\bs\nabla_h\bs\phi_h+p_h\bs I-(\bs\nabla_H\bs\phi_H+p_H\bs I), \bs\nabla_h\bs\psi_h) \\
&= (\curl\bs w_h, \bs\psi_h) - (\bs\nabla_H\bs\phi_H+p_H\bs I, \bs\nabla_h\bs\psi_h) \\
&= (\curl\bs w_h, \bs\psi_h) - (\bs\nabla_H\bs\phi_H+p_H\bs I, \bs\nabla_h(\bs I_H^{s}\bs\psi_h)) \\
&= (\curl\bs w_h, \bs\psi_h) - (\curl\bs w_H, \bs I_H^{s}\bs\psi_h).
\end{align*}
Thus we have
\begin{align}
(\bs\nabla_h(\bs\phi-\bs\phi_h), \bs\nabla_h(\bs\phi_h-\bs\phi_H))=&\,(\curl(\bs w_h-\bs w_H), \bs\psi_h) \notag\\
&+ (\curl\bs w_H, \bs\psi_h-\bs I_H^{s}\bs\psi_h). \label{eq:20200306}
\end{align}
Next we estimate the right hand side of the last equation term by term.

Due to the discrete Stokes complex \eqref{eq:Stokescomplex3dncfem} and the fact $\div_h\bs \psi_h=0$, there exists $\bs v_h\in\bs W_h$ such that $\bs\psi_h=\curl_h\bs v_h$.
% Employing \eqref{eq:errorestimateIcW}, it follows
% \[
% (\curl(\bs w_h-\bs w_H), \curl(\bs v_h-\bs I_h^c\bs v_h))\leq h\|\curl(\bs w_h-\bs w_H)\|_0|\bs\psi_h|_{1,h}.
% \]
From \eqref{eq:GalerkinOrthcurl2grid} and \eqref{eq:cd2grid}, we have
\begin{align*}
(\curl(\bs w_h-\bs w_H), \bs\psi_h)&=(\curl(\bs w_h-\bs w_H), \curl_h\bs v_h) \\
&=(\curl(\bs w_h-\bs w_H), \curl_h(\bs v_h-\bs I_H^c\bs v_h)) \\
&=(\curl(\bs w_h-\bs w_H), \curl_h\bs v_h-\bs I_H^d\curl_h\bs v_h) \\
&=(\curl(\bs w_h-\bs w_H), \bs\psi_h-\bs I_H^d\bs\psi_h).
\end{align*}
Then it holds from \eqref{eq:errorestimate2gridId} that
\begin{equation*}
(\curl(\bs w_h-\bs w_H), \bs\psi_h)\lesssim H\|\curl(\bs w_h-\bs w_H)\|_0|\bs\psi_h|_{1,h}.
\end{equation*}
Thanks to \eqref{eq:errorestimate2gridIs}, we obtain
\begin{align*}
(\curl\bs w_H, \bs\psi_h-\bs I_H^{s}\bs\psi_h)&=\sum_{K\in\mathcal T_H\backslash\mathcal T_h}(\curl\bs w_H, \bs\psi_h-\bs I_H^{s}\bs\psi_h)_K \\
&\lesssim |\bs\psi_h|_{1,h}\Bigg(\sum_{K\in\mathcal T_H\backslash\mathcal T_h}h_K^2\|\curl\bs w_H\|_{0,K}^2\Bigg)^{1/2}.
\end{align*}
Finally we conclude from \eqref{eq:20200306}, the last two inequalities and \eqref{eq:20200307-3}. 
\end{proof}

\section{Convergence}\label{sec:convergence}
In this section, we propose an adaptive algorithm (Algorithm \ref{alg:afem}) based on the estimators in Section \ref{sec:estimator}. Then its convergence is proved using the quasi-orthogonality in Section \ref{sec:quasi-orthogonality}. 
The methodology in the convergence mainly follows that of \cite{HuXu2013}. There are two major modifications: 
one first needs to control the perturbation of data for the Stokes problem, then the convergence is proved only for $\|\curl(\bs u - \bs u_h) \|_0$ and $|\bs \phi - \bs \phi_h|_{1,h}$ without the Lagrange multiplier variable. Notice that if conforming stable Stokes pairs are used, the convergence of any adaptive finite element method for Stokes equation is extremely hard as the velocity and pressure are coupled together. 

For the ease of the readers, the following short notations are adopted throughout the proof of the contraction in Theorem \ref{theorem:convergence-afem} and the lemmas needed. For $\bs V_h$ defined on $\mathcal{T}_k$, we denote them by $\bs V^c_k$ and $\bs V^{\mathrm{CR}}_k$. Similarly, the approximations on $\mathcal{T}_k$ are denoted by $\bs w_k$, $\bs \phi_k$, and $\bs u_k$ respectively. Let $h_k:= \max\limits_{K\in \mathcal{T}_k} h_K$, and denote the quantities involving two consecutive levels of meshes as follows:
% Denote $\Phi_h := (\bs w_h, \bs\phi_h)$
%
\begin{align*}
& E_k := |\bs \phi - \bs \phi_k|_{1,k} := \left(\sum_{K\in \mathcal{T}_k}\|\nabla_{k} (\bs \phi - \bs\phi_k)\|_{0,K}^2\right)^{1/2},
\\
& R_{k+1} := |\bs \phi_k - \bs \phi_{k+1}|_{1,k+1} :=\left(\sum_{K\in \mathcal{T}_{k+1}}\|\nabla_{k+1} (\bs \phi_k - \bs \phi_{k+1})\|_{0,K}^2\right)^{1/2}, 
\\
& e_k := \|\curl (\bs w - \bs w_k)\|_0, \quad 
r_{k+1} := \|\curl (\bs w_{k} - \bs w_{k+1})\|_0,
\\
&\eta_{1,k}(\mathcal{M}) := \eta_1(\bs w_k, \bs f,\mathcal M), 
\\
&\eta_{2,k}(\mathcal{M}) := \eta_2(\bs \phi_k, \bs w_k, \mathcal M),
\\
&g_k(\mathcal{M}) := \Bigg(\sum_{K\in \mathcal M}h_K^2\|\curl\bs w_k\|_{0,K}^2\Bigg)^{1/2}, 
\end{align*}
where $\nabla_k$ stands for the discrete gradient $\nabla_h$ defined piecewisely on all $K\in \mathcal{T}_k$, and $\mathcal{M}$ can be $\mathcal{T}_k$, $\mathcal{T}_{k+1}$, or $\mathcal{T}_k\backslash \mathcal{T}_{k+1}$.

\begin{algorithm}[htb]
\caption{An adaptive nonconforming finite element method}
\label{alg:afem}
\begin{algorithmic}[1]
\Require $\mathcal{T}_0$, $\bs f$, $\texttt{tol}$, $\theta_1, \theta_2 \in (0, 1)$.
\Ensure $\mathcal{T}_N$, $\bs \phi_N$, $\bs u_N$.

\State{$\eta_1 = \texttt{tol}, \eta_2 = \texttt{tol}, k = 0.$}

% \While{$\eta_1\geqslant \texttt{tol}$ and $\eta_2\geqslant \texttt{tol}$}
\While{True}

\State{\textbf{SOLVE}:}
Solve \eqref{eq:mfem1}--\eqref{eq:mfem2} and \eqref{eq:mfem3}--\eqref{eq:mfem4} on $\mathcal{T}_k$ to 
get $(\bs w_k,\bs \phi_k)$;

\State{\textbf{ESTIMATE}:}  Compute $\eta_1(\bs w_k, \bs f,K)$ and $\eta_2(\bs \phi_k, \bs w_k, K)$ for all $K\in \mathcal T_k$

\State{\qquad\qquad\qquad\quad$\eta_1 \gets \eta_1(\bs w_k, \bs f, \mathcal T_k)$}
\State{\qquad\qquad\qquad\quad$\eta_2 \gets \eta_2(\bs \phi_k, \bs w_k,  \mathcal T_k)$}

\If{$\eta_1< \texttt{tol}$ and $\eta_2< \texttt{tol}$}
\State{Break}
\EndIf

\State{\textbf{MARK}:  } Seek a minimum $\mathcal{M} \subseteq 
\mathcal{T}_k$ such that
\begin{equation}
\label{eq:mark}
\tag{M}
\begin{aligned}
 \eta^2_1(\bs w_k, \bs f,\mathcal M)& \geq 
\theta_1 \eta_1^2(\bs w_k, \bs f, \mathcal T_k)
\\
\text{ and }\;\eta_2^2(\bs w_k, \bs\phi_k, \mathcal M)
&\geq \theta_2 \eta_2^2(\bs w_k, \bs\phi_k, \mathcal T_k)
\end{aligned}
\end{equation}

\State{\textbf{REFINE}:  } Bisect $K\in \mathcal{M}$ and their neighbors to form a conforming 
$\mathcal{T}_{k+1}$;

% \State{$\eta_1 \gets \eta_1(\bs w_k, \bs f, \mathcal T_k)$}
% \State{$\eta_2 \gets \eta_2(\bs \phi_k, \bs w_k,  \mathcal T_k)$}
% \State{$\mathcal T_H \gets \mathcal T_h$}.
\State{$k\gets k+1$}
\EndWhile
\State{$N\gets k$}
\State{Solve \eqref{eq:mfem5}--\eqref{eq:mfem6} on $\mathcal{T}_N$ to get $\bs u_N$.}

\end{algorithmic}
\end{algorithm} 

The following two lemmas concern the contraction and the continuity of the estimators on two nested meshes, the proofs are standard in the AFEM literature thus omitted, the reader can refer to, e.g.,  \cite{HuXu2013,ZhongChenShuWittumEtAl2012}.
\begin{lemma}[Contraction of the estimators]
\label{lem:eta-contraction}
 Let $(\bs w_k, \bs \phi_k)\in \bs V_k^c \times\bs V_k^{\rm CR} $ 
and $(\bs w_{k+1}, \bs \phi_{k+1})\in \bs V_{k+1}^c \times\bs V_{k+1}^{\rm CR} $ be the solutions to \eqref{eq:mfem1}-\eqref{eq:mfem4} 
on triangulations $\mathcal T_k$ and $\mathcal T_{k+1}$ obtained through Algorithm \ref{alg:afem},  $\rho = 1-2^{-1/3}$ and there exists positive constants $\beta_i \in (1-\rho\theta_i,1)$ for $i=1,2$ such that 
\begin{equation*}
\begin{gathered}
\eta_{i,k}^2(\mathcal{T}_{k+1}) 
\leq \beta_i \eta_{i,k}^2(\mathcal{T}_{k}) 
-\big[\beta_i - (1 - \rho \theta_i) \big] \eta_{i,k}^2(\mathcal{T}_{k}).
\end{gathered}
\end{equation*}

\end{lemma}

\begin{lemma}[Continuity of the estimator]
\label{lem:eta-continuity}
Under the same assumption with Lemma \ref{lem:eta-contraction}
then, then given positive constants $\delta_i\in (0,1)$ for $i=1,2$,
\begin{equation*}
% \label{eq:eta-continuity}
\begin{aligned}
&\eta_{1,k+1}^2(\mathcal{T}_{k+1})
\leq (1+\delta_1) \eta_{1,k}^2(\mathcal{T}_{k+1}) + \frac{C_1}{\delta_1} r_{k+1}^2,
\\
\text{and  }& 
\eta_{2,k+1}^2(\mathcal{T}_{k+1})
\leq (1+\delta_2) \eta_{2,k}^2(\mathcal{T}_{k+1}) 
+ \frac{C_2}{\delta_2} \bigl(h_{k+1}^2 r_{k+1}^2 + R_{k+1}^2\bigr),
\end{aligned}
\end{equation*}
where $C_i$ depends on the shape-regularity of the mesh.
\end{lemma}

\begin{theorem}[Contraction]
\label{theorem:convergence-afem}
Let $(\bs w, \bs \phi)\in \bs H_0(\curl, \Omega)\times  \boldsymbol H_0^1(\Omega; \mathbb{R}^3)$ be the solutions to \eqref{eq:quarticcurldecouple1}-\eqref{eq:quarticcurldecouple4} without the Lagrange multipliers, and $(\bs w_{k+1}, \bs \phi_{k+1})\in \bs V_{k+1}^c \times\bs V_{k+1}^{\rm CR} $ and $(\bs w_k, \bs \phi_k)\in \bs V_k^c \times\bs V_k^{\rm CR} $ be their approximations in problems \eqref{eq:mfem1}-\eqref{eq:mfem4} on $\mathcal T_{k+1}$ and $\mathcal T_k$, respectively. If $\mathcal{T}_{k+1}$ is a conforming refinement from $\mathcal{T}_k$ with $h_{k+1}\leq h_k$, then there exist $\gamma_1, \gamma_2, \mu, \beta>0$, and $0<\alpha<1$ such that the AFEM in Algorithm \ref{alg:afem} satisfies
\begin{equation*}
% \label{eq:convergence-afem}
 \resizebox{0.915\textwidth}{!}{$
\begin{aligned}
& |\bs \phi - \bs \phi_{k+1}|_{1,k+1}^2 + \mu h_{k+1}^2\|\curl (\bs w - \bs w_{k+1})\|_0^2
\\
&\quad  + \gamma_1 h_{k+1}^2\eta_1^2(\bs w_{k+1}, \bs f,\mathcal T_{k+1}) 
+ \gamma_2 \tilde{\eta}_2^2(\bs\phi_{k+1},\bs w_{k+1}, \mathcal T_{k+1})
\\ 
\leq & \;\alpha\Bigl( |\bs \phi - \bs \phi_{k}|_{1,k}^2 + \mu h_{k}^2\|\curl (\bs w - \bs w_{k})\|_0^2
+ \gamma_1  h_{k}^2\eta_1^2(\bs w_{k}, \bs f,\mathcal T_{k}) 
+ \gamma_2 \tilde{\eta}_2^2(\bs\phi_k, \bs w_k, \mathcal T_{k})
\Bigr)
\end{aligned}
$}
\end{equation*}
where
\[
\tilde{\eta}_2^2(\bs w_k, \bs\phi_k,\mathcal T_{k}): = 
\bigg({\eta}_2^2(\bs w_k, \bs\phi_k,\mathcal T_{k}) + \beta \sum_{K\in \mathcal T_k}h_K^2\|\curl\bs w_k\|_{0,K}^2\bigg)
\]
is the modified estimator.
\end{theorem}

\begin{proof}
First by the Galerkin orthogonality \eqref{eq:GalerkinOrthcurl}, we have $e_{k+1}^2 = e_k^2 - r_{k+1}^2$. Assuming that the constant in the quasi-orthogonality \eqref{eq:phiHphihORerror} is $\sqrt{C_Q/2}$, we have by Young's inequality for an $\epsilon\in (0,1)$
\begin{align}
E_{k+1}^2 
&= E_k^2 - R_{k+1}^2  
- 2\bigl(\nabla_{k+1}(\bs \phi - \bs \phi_{k+1}), \nabla_{k+1}(\bs \phi_{k+1} - \bs \phi_k)\bigr) \notag
\\
& \leq E_k^2 - R_{k+1}^2 +\sqrt{2C_Q} \bigl(h_k r_{k+1} 
+g_k(\mathcal{T}_k\backslash \mathcal{T}_{k+1}) \bigr)E_{k+1}
  \notag\\
& \leq E_k^2 - R_{k+1}^2 + \epsilon E_{k+1}^2 
+ \frac{C_Q}{\epsilon} \Big(h_k^2 r_{k+1}^2
+ g_k^2(\mathcal{T}_k\backslash \mathcal{T}_{k+1}) \Big). \label{eq:20210820}
\end{align}
Now from Lemmas \ref{lem:eta-contraction} and \ref{lem:eta-continuity}, and $h_{k+1}\leq h_k$, we choose $\delta_i$ ($i=1,2$) such that $\beta_i := (1-\rho \theta_i) (1+\delta_i)\in (0,1)$, we have

\begin{equation}
\label{eq:est-eta1}
\eta_{1,k+1}^2(\mathcal{T}_{k+1}) \leq \beta_1 \eta^2_{1,k}(\mathcal{T}_k)
+ \Big[\delta_1 \beta_1 - (1+\delta_1)
\big(\beta_1 - (1-\rho\theta_1)\big)\Big] \eta^2_{1,k}(\mathcal{T}_k)
+ \frac{C_1}{\delta_1} r_{k+1}^2,
\end{equation}
and
\begin{equation}
\label{eq:est-eta2}
\begin{aligned}
\eta_{2,k+1}^2(\mathcal{T}_{k+1}) \leq
&\; \beta_2 \eta^2_{2,k}(\mathcal{T}_k)
+ \Big[\delta_2 \beta_2 - (1+\delta_2)
\big(\beta_2 - (1-\rho\theta_2)\big)\Big] \eta^2_{2,k}(\mathcal{T}_k)
\\
& \quad + \frac{C_2}{\delta_2} \bigl(h_k^2 r_{k+1}^2 + R_{k+1}^2\bigr).
\end{aligned}
\end{equation}
Next for the element residual term in $\eta_2$ on each $K$ we have:
\[
\begin{aligned}
\|\curl \bs w_{k+1}\|_{0,K}^2  = &\;
\| \curl \bs w_k\|_{0,K}^2 - 
\| \curl (\bs w_k - \bs w_{k+1})\|_{0,K}^2
\\ 
&\quad - 2\bigl(\curl \bs w_{k+1}, \curl (\bs w_k - \bs w_{k+1}) \bigr)_K.
\end{aligned}
\]
By the Young's inequality for a $\delta_3\in (0,1)$,
\[
(1-\delta_3)\|\curl \bs w_{k+1}\|_{0,K}^2 \leq \| \curl \bs w_k\|_{0,K}^2 + 
\frac{1-\delta_3}{\delta_3} \| \curl (\bs w_k - \bs w_{k+1})\|_{0,K}^2,
\]
and consequently applying similar techniques with Lemma \ref{lem:eta-contraction} yields:
\begin{equation}
\label{eq:est-g}
(1-\delta_3) g_{k+1}^2(\mathcal{T}_{k+1}) \leq g_k^2(\mathcal{T}_{k})
-\rho g_k^2(\mathcal{T}_k\backslash \mathcal{T}_{k+1}) + \frac{C_3}{\delta_3}h_{k+1}^2 r_{k+1}^2.  
\end{equation}

To prove the contraction result, we define 
\[
\begin{aligned}
\mathfrak{G}_{k+1} :=&\; (1-\epsilon) E_{k+1}^2 + \mu h_{k+1}^2 e_{k+1}^2
\\
& + 
  \gamma_1 h_{k+1}^2\eta_{1,k+1}^2(\mathcal{T}_{k+1}) 
+  \gamma_2 \eta_{2,k+1}^2(\mathcal{T}_{k+1})  + (1-\delta_3)\gamma_3 g_{k+1}^2(\mathcal{T}_{k+1})
\end{aligned}
\]
and
\[
\begin{aligned}
\overline{\mathfrak{G}}_{k} :=&\;E_{k}^2 + \mu h_{k}^2 e_{k}^2
\\
& + 
  \gamma_1 \beta_1 h_{k}^2\eta_{1,k}^2(\mathcal{T}_{k}) 
+  \gamma_2 \beta_2 \eta_{2,k}^2(\mathcal{T}_{k})  + \gamma_3 g_{k}^2(\mathcal{T}_{k}).
\end{aligned}
\]
Combining estimates \eqref{eq:20210820}, \eqref{eq:est-eta1}, \eqref{eq:est-eta2} and \eqref{eq:est-g} above, 
\begin{equation}
\label{eq:est-contract}
  \begin{aligned}
\mathfrak{G}_{k+1} & \leq 
\overline{\mathfrak{G}}_{k} -\left(1-\frac{\gamma_2 C_2}{\delta_2} \right)R_{k+1}^2 
\\
&\quad 
- \left(\mu -  \frac{C_Q }{\epsilon}-\sum_{j=1}^3 \frac{\gamma_j C_j}{\delta_j} \right)
h_k^2 r_{k+1}^2
\\
& \quad 
+\gamma_1 h_k^2 \Big[\delta_1 \beta_1 - (1+\delta_1)
\big(\beta_1 - (1-\rho\theta_1)\big)\Big] \eta^2_{1,k}(\mathcal{T}_k)
\\
& \quad +\gamma_2 \Big[\delta_2 \beta_2 - (1+\delta_2)
\big(\beta_2 - (1-\rho\theta_2)\big)\Big] \eta^2_{2,k}(\mathcal{T}_k)
\\
& \quad -\left( {\gamma_3 \rho} - \frac{C_Q}{\epsilon} \right)
g_k^2(\mathcal{T}_k\backslash \mathcal{T}_{k+1}).
\end{aligned}
\end{equation}
The constants are formulated such that all terms on the right hand side except the first in the inequality above vanish for a fixed $\epsilon\in (0,1)$, which is determined later. To this end, aside from the choice of $\delta_i$ ($i=1,2$) above, we choose $\gamma_2$ and $\gamma_3$ to be
\begin{equation}
\label{eq:est-gamma3}
\gamma_2 = \frac{\delta_2}{C_2}, \quad \text{and }\quad \gamma_3 = \frac{C_Q}{\rho \epsilon},
\end{equation}
and $\gamma_1$ is a free constant. Additionally, $\mu$ is free as well, and is to be chosen sufficiently large such that the following holds regardless of what values $\epsilon$ and $\delta_3$ take
\begin{equation*}
% \label{eq:est-mu}
\mu -  \frac{C_Q }{\epsilon}-\sum_{j=1}^3 \frac{\gamma_j C_j}{\delta_j} \geq 0.
\end{equation*} 

% We can choose appropriate $\gamma_i$ ($i=1,2,3$) and $\beta_j$ ($j=1,2$) such that:
% \[
% \begin{aligned}
% & \gamma_2 = \frac{\epsilon}{C_E}, \quad  \gamma_1 \geq \frac{4C_Q^2 H^2}{\epsilon}+\frac{\gamma_2 C_E}{\epsilon},
% \\
% & \beta_1 = (1-\rho\theta) (1+\epsilon), \quad \beta_2 = (1-\rho\theta/2) (1+\epsilon),
% \\
% \text{and } &\;  \gamma_3\geq \frac{8C_Q^2}{\rho\epsilon}.
% \end{aligned}
% \]
Consequently, \eqref{eq:est-contract} becomes $\mathfrak{G}_{k+1}  \leq 
\overline{\mathfrak{G}}_{k}$. For an $\alpha\in (0,1)$ we rewrite the right hand side of \eqref{eq:est-contract} as 
\[
\overline{\mathfrak{G}}_{k} = \alpha\mathfrak{G}_{k} + \mathfrak{R}_k,
\]
where 
\[
\begin{aligned}
\mathfrak{R}_k:=& \big(1-\alpha(1-\epsilon)\big) E_k^2 + \mu(1-\alpha)h_k^2 e_k^2
\\
& + \gamma_1(\beta_1 - \alpha)h_k^2 \eta_{1,k}^2(\mathcal{T}_k)
+ \gamma_2(\beta_2 - \alpha)\eta_{2,k}^2(\mathcal{T}_k)
+ \gamma_3\bigl(1- \alpha(1-\delta_3)\bigr) g_k^2(\mathcal{T}_k).
\end{aligned}
\]
It suffices to show that for the constants chosen and to be determined, there exists an $\alpha\in (0,1)$ such that $\mathfrak{R}_k\leq 0$. 
Now assuming that the constant in the reliability estimate \eqref{eq:errorposterioriphihph} is $\sqrt{C_R/2}$, and the fact that $g_k(\mathcal{T}_k)\leq \eta_{2,k}(\mathcal{T}_k) $ we have
\begin{equation*}
%  \resizebox{0.90\textwidth}{!}{$
\begin{aligned}
\mathfrak{R}_k\leq \bigl(h_k^2 \eta_{1,k}^2(\mathcal{T}_k) + \eta_{2,k}^2(\mathcal{T}_k) \bigr)   & \Big[C_R\big(1-\alpha(1-\epsilon) + \mu(1-\alpha)\big) 
\\
&\; + (\beta_1 - \alpha)\gamma_1 + (\beta_2 - \alpha)\gamma_2 + \gamma_3 \big(1-\alpha(1-\delta_3)\big)\Big] .   
\end{aligned}
% $}
\end{equation*}
It is straightforward to verify the second term on the right hand side above vanishes when letting 
\begin{equation}
\label{eq:est-alpha}
\alpha := \frac{C_R(1+\mu)+\gamma_1\beta_1+\gamma_2\beta_2+\gamma_3}
{C_R(1+\mu-\epsilon)+\gamma_1+\gamma_2+\gamma_3(1-\delta_3)} > 0.
\end{equation}
To make $\alpha<1$, by \eqref{eq:est-gamma3}, the following inequality is needed:
\begin{equation}
\label{eq:est-eps}
\epsilon C_R + \delta_3 \frac{C_Q}{\rho \epsilon} < \gamma_1(1-\beta_1)+\gamma_2(1-\beta_2).
\end{equation}
Choosing 
\[
\delta_3 = \rho\epsilon^2\quad \text{and} \quad \gamma_1 = \frac{\gamma_2\beta_2}{1-\beta_1},
\]
we have \eqref{eq:est-eps} holds with 
\[
0<\epsilon <\min \left\{\frac{\delta_2}{C_2(C_R+C_Q)},1\right\}
\]
and thus $\alpha$ defined in \eqref{eq:est-alpha} is in $(0,1)$. As a result, $\mathfrak{G}_{k+1} \leq \alpha \mathfrak{G}_{k}$ is shown, and finally the theorem follows by acknowledging that $g_k(\mathcal{T}_k)$ is a part of $\eta_{2,k}(\mathcal{T}_k)$.
\end{proof}

% \begin{remark}[Optimality]
% With the contraction result in Theorem \ref{theorem:convergence-afem}, the optimality simply follows from the quasi-orthogonality, and the existing results on the discrete reliability for the nonconforming $P_1$--$P_0$ pair (\cite[Lemma 5.1]{HuXu2013}) and the lowest order N\'ed\'elec elements(\cite[Theorem 5.3]{ZhongChenShuWittumEtAl2012}).
% \end{remark}

\section{Numerical Examples}
\label{sec:numerics}
The numerical experiments in this section, as well as in Section \ref{sec:numerics-uniform}, are carried out using $i$FEM \cite{Chen:2009ifem}. The code used for this paper is publicly available at \url{https://github.com/lyc102/ifem/tree/master/research/quadCurl}. The linear systems for the Stokes problem originated from \eqref{eq:mfem3}--\eqref{eq:mfem4} are solved using MINRES with a diagonal preconditioner, the inverse of the Schur complement of the lower right block can be efficiently approximated by a few V-cycles. For the discretized Maxwell saddle problems \eqref{eq:mfem1}--\eqref{eq:mfem2} and \eqref{eq:mfem5}--\eqref{eq:mfem6}, we use the multigrid method for Hodge Laplacian to solve a block factorization; see \cite[Section 4.4]{ChenWuZhongZhou2018Multigrid}. 

We mainly compare the performance of the adaptive algorithm under the proposed separate marking strategy \eqref{eq:mark} in Algorithm \ref{alg:afem} versus two single marking strategies: A minimum $\mathcal{M} \subseteq \mathcal{T}_k$ is sought such that
\begin{equation}
\label{eq:mark-single1}
\tag{M1}
{\eta}^2(\bs w_h, \bs\phi_h, \bs f,\mathcal M)
\geq \theta {\eta}^2(\bs w_h, \bs\phi_h, \bs f,\mathcal T_k),
\end{equation}
or
\begin{equation}
\label{eq:mark-single2}
\tag{M2}
\widetilde{\eta}^2(\bs w_h, \bs\phi_h, \bs f,\mathcal M)
\geq \theta \widetilde{\eta}^2(\bs w_h, \bs\phi_h, \bs f,\mathcal T_k),
\end{equation}
where $\eta$ is defined as 
\begin{equation}
\label{eq:eta}
\eta^2(\bs w_h, \bs\phi_h, \bs f,\mathcal M):=\eta_1^2(\bs w_h, \bs f,\mathcal M_h)+\eta_2^2(\bs\phi_h, \bs w_h,\mathcal M),
\end{equation}
and $\widetilde{\eta}(\bs w_h, \bs\phi_h, \bs f,\mathcal M)$ is defined by the $L^2$-sum of the weighted $\eta_1$ and $\eta_2$:
\begin{equation}
\label{eq:eta-scaled}
\resizebox{0.915\textwidth}{!}{$
\begin{aligned}
\widetilde{\eta}^2(\bs w_h, \bs\phi_h, \bs f,\mathcal M)
& := \sum_{K\in\mathcal M}h_K^4\|\bs f\|_{0,K}^2+\sum_{F\in\mathcal F_h^i(\mathcal M)}
h_F^3 \|\llbracket(\curl\bs w_h)\times\bs n_F\rrbracket\|_{0,F}^2
\\
&\; + \sum_{K\in\mathcal M}h_K^2\|\curl\bs w_h\|_{0,K}^2+\sum_{F\in\mathcal F_h(\mathcal M)}h_F\|\llbracket\bs n_F\times(\bs\nabla_h\bs \phi_h)\rrbracket\|_{0,F}^2.
\end{aligned}
$}
\end{equation} 
On a uniform mesh, $\widetilde{\eta}^2$ can be viewed as approximately $h^2 \eta_1^2 + \eta_2^2$.

To demonstrate the reason why we opt for the proposed separate marking strategy \eqref{eq:mark}, and not single markings such as \eqref{eq:mark-single1} and \eqref{eq:mark-single2}, we construct a toy example using $\bs{u} = \curl \langle 0, 0, \mu\rangle$ for a potential function $\mu = r^{19/6}\sin(2\theta/3)$ in the cylindrical coordinate as in Section \ref{sec:numerics-uniform} example 2 on an L-shaped domain. In this case we have a regular 
$\bs \phi = \curl \bs u\in \bs H^{13/6-\epsilon}(\Omega)$ while $\bs w$ has a mild singularity near the nonconvex corner.

The results for the convergence of Algorithm \ref{alg:afem} using marking strategies \eqref{eq:mark}, \eqref{eq:mark-single1}, and \eqref{eq:mark-single2} can be found in Figures \ref{fig:ex3-conv}, \ref{fig:ex3-conv-single1}, and \ref{fig:ex3-conv-single2}, respectively. 

Using the proposed marking \eqref{eq:mark}, we obtain the desired optimal convergence for $|\bs\phi_h -\curl \bs u|_{1,h}$ being optimal in that the convergence is at the rate of ``linear'' $\approx\# (\mathrm{DoF})^{-1/3}$. Additionally, $\|\bs\phi_h -\curl \bs u\|_{0}$ and $\|\bs u - \bs u_h\|_0$ converge ``quadratically'', i.e., in the order of approximately $\# (\mathrm{DoF})^{-2/3}$ even though the error estimator's reliability and efficiency are not directly measured in those norms. In this experiment, $\theta_1 = 0.5$ and $\theta_2 = 0.3$.

Using marking \eqref{eq:mark-single1} with $\theta=0.3$, if $\eta_1$ is unweighted in the $\eta$ in \eqref{eq:eta}, due to $\bs w$ being singular while $\bs \phi$ being regular, thanks to $\eta_1$ being locally efficient, the marked elements are dominantly concentrated on which $\eta_1$ are large. As a result, marking \eqref{eq:mark-single1} drives the AFEM algorithm favoring reducing the error for $\bs w$, while the errors in approximating $\bs \phi$ and $\bs u$ barely change (Figure \ref{fig:ex3-conv-single1}). Fortunately, due to the regularity lifting effect from $\bs w_h$ being the data for the problem of $\bs \phi_h$ (cf. Lemma \ref{lem:estimate-phitilde}), to achieve the optimal rate of convergence, $\bs w_h$ does not have to be approximated to the same precision with $\bs \phi_h$.

If marking \eqref{eq:mark-single2} with $\theta =0.3$ is used where $\eta_1$ is locally weighted by the mesh size $h_K$, the optimal rates of convergence is restored. However, one does benefit from choosing different marking parameters for approximating $\bs w$ and $\bs \phi$ due to the regularity difference (see e.g., \cite{Chen.Dai:2002efficiency}). Moreover, one does not have the Galerkin orthogonality for 
$\|\curl (\bs w - \bs w_h)\|$ to exploit in the proof of the contraction in Theorem \ref{theorem:convergence-afem}, because the consecutive difference is measured under a norm weighted by the mesh size. As a result, it needs new tools that are not available in any of the current literature to prove a similar contraction result when the local error indicator is further weighted by the local mesh size $h_K$.

\begin{figure}[htbp]
  \centering
\begin{subfigure}[b]{0.31\linewidth}
    \includegraphics[width=1\textwidth]{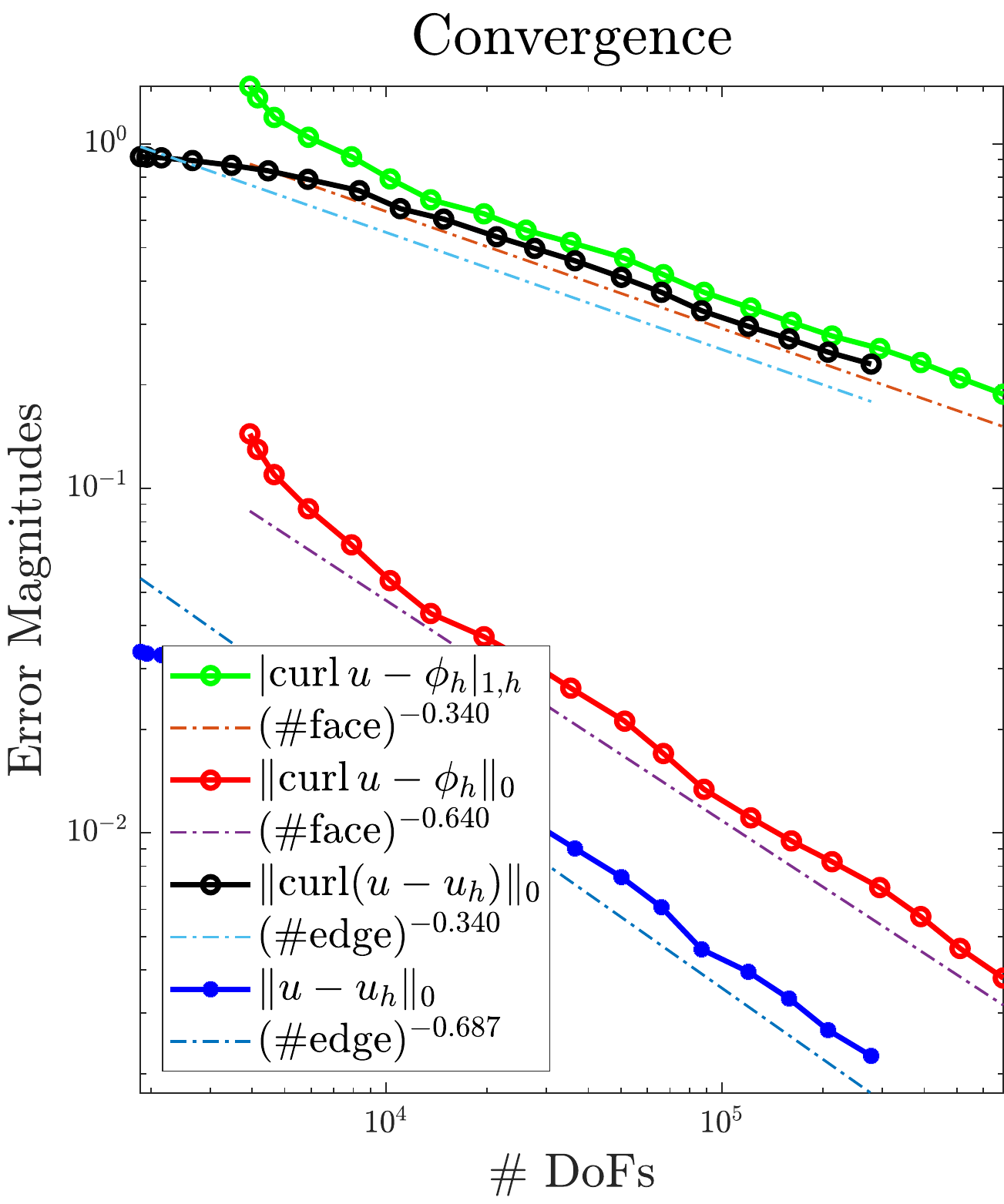}
    \caption{\label{fig:ex3-conv}}
\end{subfigure}%
\hspace{3pt}
\begin{subfigure}[b]{0.3\linewidth}
  \centering
      \includegraphics[width=\textwidth]{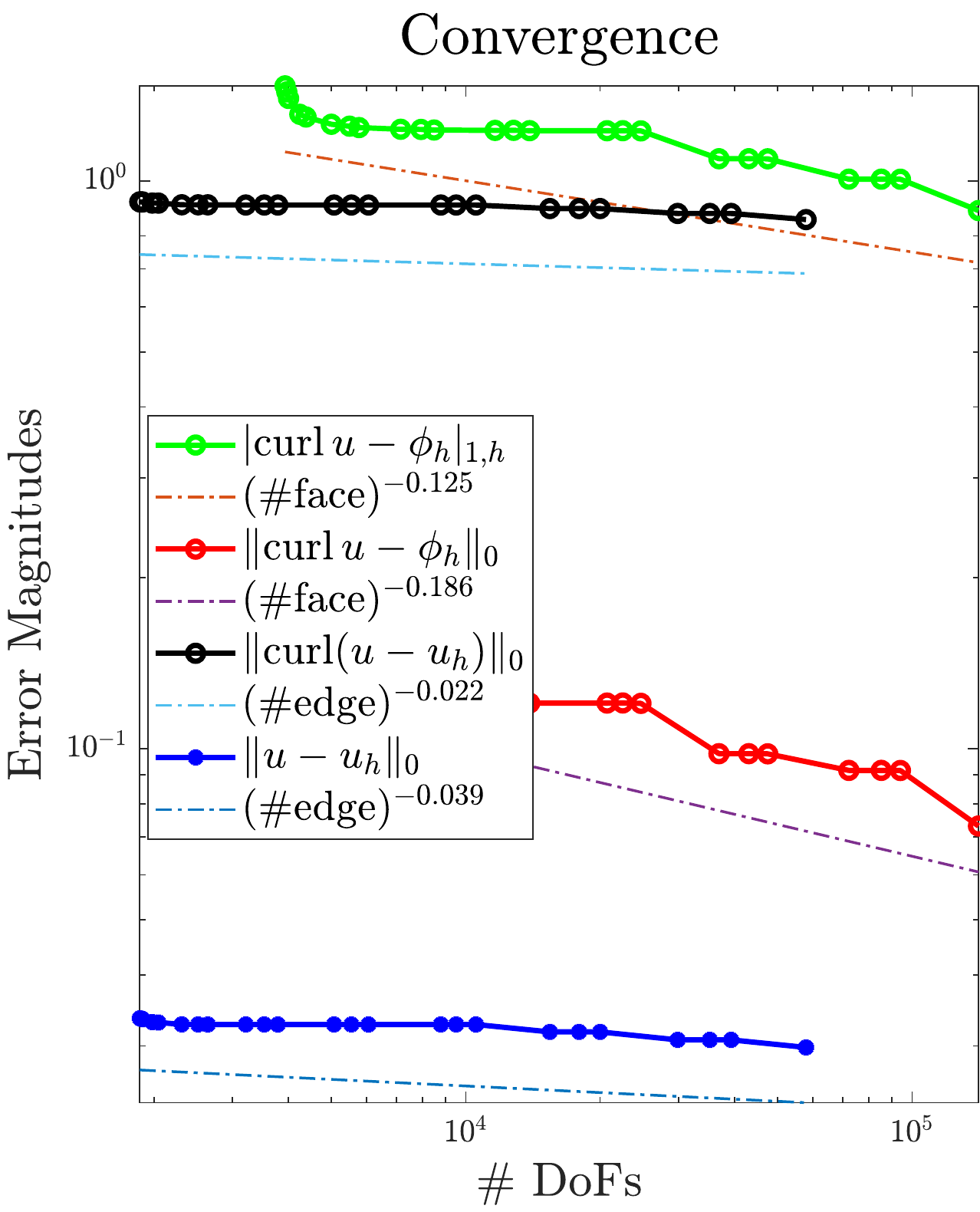}
      \caption{\label{fig:ex3-conv-single1}}
\end{subfigure}
\hspace{3pt}
\begin{subfigure}[b]{0.3\linewidth}
      \includegraphics[width=\textwidth]{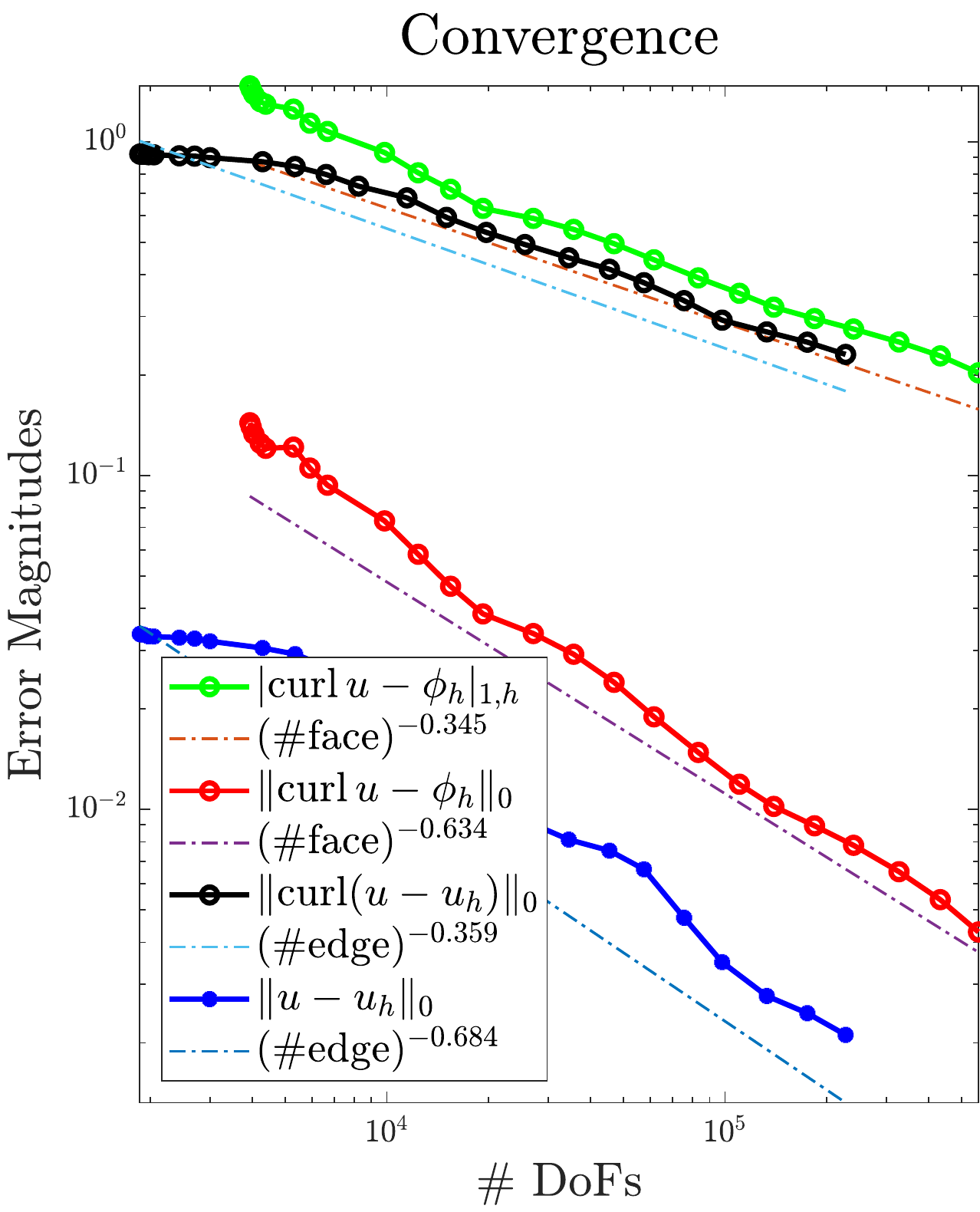}
      \caption{\label{fig:ex3-conv-single2}}
\end{subfigure}
\caption{The convergence results for AFEM \ref{alg:afem} using different marking strategies: (\subref{fig:ex3-conv}) separate marking for $\eta_1$ and $\eta_2$. (\subref{fig:ex3-conv-single1}) single marking for $\eta$. (\subref{fig:ex3-conv-single2}) single marking for $\widetilde{\eta}$. }
\end{figure}

\section*{Acknowledgement}
We greatly appreciate the anonymous reviewers' revising suggestions.

\end{document}